\documentclass{amsart}
\usepackage{amsmath,mathrsfs}
\usepackage{amssymb}
\usepackage{amscd}
\usepackage{graphicx}
\usepackage{textcomp}
\usepackage{stmaryrd}
\usepackage{everysel}
\usepackage{proof}
\usepackage{tikz}
\usetikzlibrary{patterns, arrows}
\usepackage{bussproofs}

\newtheorem{theorem}{Theorem}[section]
\newtheorem{lemma}[theorem]{Lemma}
\theoremstyle{definition}
\newtheorem{definition}[theorem]{Definition}
\newtheorem{proposition}[theorem]{Proposition}
\newtheorem{corollary}[theorem]{Corollary}
\newtheorem{example}[theorem]{Example}

\theoremstyle{remark}

\numberwithin{equation}{section}

\newcommand{\mf}{\mathfrak}

\newcommand{\mc}{\mathcal}

\newcommand{\msf}{\mathsf}
\newcommand{\mrm}{\mathrm}
\newcommand{\la}{\langle}
\newcommand{\ra}{\rangle}
\newcommand{\imp}{\rightarrow}

\newcommand{\Imp}{\Rightarrow}
\newcommand{\sub}{\subseteq}

\newcommand{\setof}[1]{\{\,{#1}\,\}}

\newtheorem{fact}[theorem]{Fact}

\renewcommand{\lnot}{\mathop{\sim}}

\begin{document}

\author{Minghui Ma}
\address{Institute of Logic and Cognition, Sun Yat-Sen University}
\curraddr{Xingang Xi Road 135, Guangzhou, 510275, China}
\email{mamh6@mail.sysu.edu.cn}

\author{Fei Liang}
\address{Institute of Logic and Cognition, Sun Yat-Sen University}
\curraddr{Xingang Xi Road 135, Guangzhou, 510275, China}
\email{liangf25@mail2.sysu.edu.cn}
\title{Sequent Calculi for Semi-De Morgan and De Morgan Algebras}
\maketitle

\begin{abstract}
A contraction-free and cut-free sequent calculus $\msf{G3SDM}$ for semi-De Morgan algebras, and a structural-rule-free and single-succedent sequent calculus $\msf{G3DM}$ for De Morgan algebras are developed. The cut rule is admissible in both sequent calculi. 
Both calculi enjoy the decidability and Craig interpolation. The sequent calculi are applied to prove some embedding theorems: $\msf{G3DM}$ is embedded into $\msf{G3SDM}$ via G\"odel-Gentzen translation. $\msf{G3DM}$ is embedded into a sequent calculus for classical propositional logic. $\msf{G3SDM}$ is embedded into the sequent calculus $\msf{G3ip}$ for intuitionistic propositional logic.
\end{abstract}

\section{Introduction} 

De Morgan algebras (also called "quasi-Boolean algebras"), which are (not necessarily bounded) distributive lattices with a De Morgan negation, were originally introduced by Bialynicki-Birula and Rasiowa \cite{bial:repr57}. This type of algebras was investigated by Moisil etc. under the term {\em De Morgan lattices}, and by Kalman under the term {\em distributive i-lattices} (cf.~\cite[pp.44-48]{rasi:alge74}). They have been widely discussed in the literature on universal algebra (cf.~\cite{rasi:alge74,dunn:rela82,balb:dist11}), relevance logic (\cite{ande: rele75}), Belnap's four-valued logic and the logic of bilattices (\cite{arie:reas96,bou:bial11}), rough set theory and pre-rough algebras (\cite{saha:prer14}) etc., from algebraic and logical perspectives.

Semi-De Morgan algebras were originally introduced by Sankappanavar \cite{sank:semi87}, as a common abstraction of De Morgan algebras and distributive pseudo complemented lattices. The variety of (bounded) De Morgan algebras becomes a subvariety of semi-De Morgan algebras. Hobby \cite{hobb:semi96} presented a duality theory for semi-De Morgan algebras based on Priestly duality for distributive lattices. Some subvarieties of semi-De Morgan algebras are also studied in literature (cf.~\cite{palm:semi03}).
However, the proof-theoretic aspects of semi-De Morgan and De Morgan algebras have not been well-developed.

The aim of the present paper is to investigate semi-De Morgan and De Morgan algebras from the proof-theoretic perspective. There are some 
thoughts on sequent calculus for De Morgan algebras in literature. A sequent calculus for De Morgan algebras was given in \cite{saha:prer14}, but it lacks cut elimination. 
One idea to give a cut-free sequent calculus for De Morgan algebras was presented in the sequent calculus for the basic logic of bilattices in Arieli and Avron \cite{arie:reas96}. This calculus is multi-succedent, and it contains structural rules of exchange and contraction as well as rules for the combination of negation with all connectives. 
The same idea is used by Avron \cite{avro:nega99} to study the basic systems of negation. Our formulation of sequent calculi for semi-De Morgan and De Morgan algebras will enrich this idea.

We shall first introduce a contraction-free and cut-free sequent calculus $\msf{G3SDM}$ for semi-De Morgan algebras, in the style of $\msf{G3ip}$ for intuitionistic logic (cf.~\cite{negr:stru01}), i.e., a calculus without structural rules of weakening, contraction and cut. We shall further introduce a structural-rule-free and single-succedent sequent calculus $\msf{G3DM}$ for De Morgan algebras.
These Gentzen sequent calculi allow us to prove several logical properties including the cut admissibility, decidability, and Craig interpolation. Moreover, they are applied to prove some embedding theorems. We shall prove that $\msf{G3DM}$ is embedded into $\msf{G3SDM}$ via G\"odel-Gentzen translation, and it is also embedded into the sequent calculus $\msf{G3ip}$+\textbf{Gem-at} for classical propositional logic. Furthermore, $\msf{G3SDM}$ is embedded into the sequent calculus $\msf{G3ip}$ for intuitionistic propositional logic.

This paper is organized as below. Section 2 recalls some preliminaries on semi-De Morgan and De Morgan algebras. Section 3 presents the sequent calculus $\msf{G3SDM}$ for semi-De Morgan algebras, and proves the decidability and Craig interpolation of $\msf{G3SDM}$. Section 4 introduces the sequent calculus $\msf{G3DM}$ for De Morgan algebras, and proves the decidability and Craig interpolation of $\msf{G3DM}$.
Section 5 proves some embedding theorems. Section 6 gives the conclusion.

\section{Preliminaries}
We recall some basic concepts on semi-De Morgan algebras from \cite{sank:semi87}.

\begin{definition}
An algebra $\mf{A} = (A,\vee,\wedge,\lnot, 0, 1)$ is a {\em De Morgan algebra}, if $(A, \vee, \wedge, 0, 1)$ is a bounded distributive lattice, $\lnot 0 = 1$ and $\lnot 1 = 0$, and the following conditions hold for all $a,b\in A$:
\[
\lnot(a \vee b) = \lnot a \wedge \lnot b,  \quad 
\lnot(a \wedge b) = \lnot a \vee \lnot b, \quad
\lnot\lnot a = a.
\]
The variety of all De Morgan algebras is denoted by $\msf{DM}$. 
\end{definition}

\begin{definition}
An algebra $\mf{A} = (A,\vee,\wedge,\lnot, 0, 1)$ is a {\em semi-De Moragan algebra}, if $(A, \vee, \wedge, 0, 1)$ is a bounded distributive lattice, $\lnot 0 = 1$ and $\lnot 1 = 0$, and the following conditions hold for all $a,b\in A$:
\[
\lnot(a \vee b) = \lnot a \wedge \lnot b, 
\quad
\lnot\lnot (a \wedge b) = \lnot\lnot a \wedge \lnot\lnot b,
\quad
\lnot\lnot\lnot a = \lnot a.
\]
The variety of all semi-De Morgan algebras is denoted by $\msf{SDM}$.
\end{definition}

\begin{fact}[\cite{sank:semi87}] 
A semi-De Morgan algebra $\mf{A}$ is a De Morgan algebra if and only if $\mf{A}$ satisfies the identity $a\vee b = \lnot(\lnot a\wedge \lnot b)$.
\end{fact}

Let $\Xi=\setof{p_n\mid n\in \omega}$ be a denumerable set of propositional variables.
The set of all {\em terms} $\mc{T}$ generated by $\Xi$ is defined inductively as follows:
\begin{center}
$\mc{T}\ni\varphi ::=  p \mid \bot \mid \lnot\varphi \mid (\varphi\wedge\varphi) \mid (\varphi\vee \varphi)$, where $p\in \Xi$.
\end{center}

The algebra $\mf{T} = (\mc{T}, \vee, \wedge, \lnot, \bot)$ is called the {\em term algebra} with basis $\Xi$. Define $\top:=\lnot\bot$. A term $\varphi$ is called {\em atomic} if $\varphi\in \Xi\cup\{\bot\}$. All terms are denoted by $\varphi, \psi, \chi$ etc. with or without subscripts. The {\em complexity} of a term $\varphi$ is the number of all occurrences of connectives $\lnot$, $\wedge$ and $\vee$ in $\varphi$. A {\em basic sequent} is an expression $\varphi\Imp\psi$ where $\varphi, \psi\in \mc{T}$. 

Given a semi-De Morgan algebra $\mf{A}=(A, \vee, \wedge, \lnot, 0, 1)$, an {\em assignment} in $\mf{A}$ is a function $\sigma: \Xi\imp A$. For any term $\varphi$ and assignment $\sigma$ in $\mf{A}$, let $\varphi^\sigma$ be the image of $\varphi$ under the homomorphism from the term algebra $\mf{T}$ to $\mf{A}$ which extends $\sigma$. A basic sequent $\varphi\Imp\psi$ is {\em valid} in $\mf{A}$, notation $\mf{A}\models\varphi\Imp\psi$, if $\varphi^\sigma\leq \psi^\sigma$ for any assignment $\sigma$ in $\mf{A}$, where $\leq$ is the lattice order. 
A basic sequent is {\em valid} in a class of algebras if it is valid in all its members.

\begin{definition}
The basic sequent calculus for semi-De Morgan algebras $S_\msf{SDM}$ consists of the following axioms and inference rules:
\begin{itemize}
\item Axioms:
\[
(Id)~\varphi\Imp \varphi
\quad
(D)~\varphi\wedge(\psi\vee\chi)\Imp(\varphi\wedge\psi)\vee(\varphi\wedge\chi)\quad 
(\bot)~\bot\Imp\varphi
\]
\[
(\lnot\bot)~\varphi\Imp \lnot\bot
\quad
(\lnot\lnot\bot)~\lnot\lnot\bot\Imp\varphi
\]
\[
(\lnot1)~\lnot\lnot\lnot\varphi \Rightarrow \lnot\varphi
\quad
(\lnot2)~\lnot\varphi \Rightarrow \lnot\lnot\lnot\varphi
\]
\[
(\lnot\vee)~\lnot\varphi\wedge\lnot\psi\Rightarrow\lnot(\varphi\vee\psi)
\quad
(\lnot\wedge)~\lnot\lnot\varphi\wedge\lnot\lnot\psi\Rightarrow\lnot\lnot(\varphi\wedge\psi)
\]
\item Rules for lattice operations:
\[
\AxiomC{$\varphi_i\Imp\psi$}
\RightLabel{ $(\wedge\Imp)(i=1,2)$}
\UnaryInfC{$\varphi_1\wedge\varphi_2\Imp \psi$}
\DisplayProof
\quad
\AxiomC{$\varphi \Imp\psi$\quad $\varphi\Imp\chi$}
\RightLabel{ $(\wedge\Imp)$}
\UnaryInfC{$\varphi \Imp \psi\wedge\chi$}
\DisplayProof
\]
\[
\AxiomC{$\varphi \Imp\chi$\quad $\psi\Imp\chi$}
\RightLabel{ $(\vee\Imp)$}
\UnaryInfC{$\varphi\vee\psi\Imp\chi$}
\DisplayProof
\quad
\AxiomC{$\varphi\Imp\psi_i$}
\RightLabel{ $(\vee\Imp)(i=1,2)$}
\UnaryInfC{$\varphi \Imp \psi_1\vee\psi_2$}
\DisplayProof
\]
\item Cut and contraposition rules:
\[
\AxiomC{$\varphi \Imp\psi$\quad $\psi\Imp\chi$}
\RightLabel{ $(Cut)$}
\UnaryInfC{$\varphi \Imp\chi$}
\DisplayProof
\quad
\AxiomC{$\varphi \Rightarrow \psi$}
\RightLabel{\small $(Ctp)$}
\UnaryInfC{$\lnot\psi \Rightarrow \lnot\varphi$}
\DisplayProof
\]
\end{itemize}

The basic sequent calculus for De Morgan algebras $S_\msf{DM}$ is obtained from $S_\msf{SDM}$ by adding the axioms $\varphi\vee\psi \Rightarrow \lnot(\lnot\varphi\wedge\lnot\psi)$ and $\lnot(\lnot\varphi\wedge\lnot\psi)\Rightarrow\varphi\vee\psi$.

A basic sequent $\varphi\Imp\psi$ is {\em derivable} in a calculus $S$, notation $S\vdash \varphi\Imp\psi$, if there is a derivation of $\varphi\Imp\psi$ in $S$, i.e., a proof tree with root $\varphi\Imp\psi$, starting from axioms and using rules.
\end{definition}

\begin{fact}
For any terms $\varphi$ and $\psi$, the sequents $\lnot(\varphi\vee\psi)\Rightarrow \lnot\varphi\wedge\lnot\psi$ and $\lnot\lnot(\varphi\wedge\psi)\Rightarrow\lnot\lnot\varphi\wedge \lnot\lnot\psi$ are derivable in $S_\msf{SDM}$.
\end{fact}

\begin{theorem}
For any basic sequent $\varphi\Imp\psi$, 
(1) $S_\msf{SDM}\vdash\varphi\Imp\psi$ if and only if $ \msf{SDM}\models \varphi\Imp\psi$; (2) $S_\msf{DM}\vdash\varphi\Imp\psi$ if and only if $ \msf{DM}\models \varphi\Imp\psi$.
\end{theorem}
\begin{proof}
The soundness is shown by induction on the derivation of a basic sequent. The completeness is shown by the standard Lindenbaum-Tarski construction. Note that the equivalence relation $\equiv_{S}$ over $\mc{T}$ for $S\in\{S_\msf{SDM},S_\msf{DM}\}$ is defined by:
$\varphi\equiv_S\psi$ if and only if $S\vdash\varphi\Rightarrow\psi$ and $S\vdash\psi\Rightarrow\varphi$.
One can easily prove that $\equiv_S$ is a congruence relation. The Lindenbaum-Tarski algebra is then defined, and the completeness is obtained.
\end{proof}

\section{A Sequent Calculus for Semi-De Morgan Algebras} 
In this section, we shall present a Gentzen-style sequent calculus $\msf{G3SDM}$ for semi-De Morgan algebras. Then we shall show the completeness of $\msf{G3SDM}$ with respect to the variety of semi-De Morgan algebras, the decidability of derivability, and Craig interpolation of $\msf{G3SDM}$.

\subsection{The sequent calculus $\msf{G3SDM}$}

A {\em basic SDM-structure} is an expression of the form $\varphi$ or $*\varphi$ where $\varphi$ is a term. 
Here $*$ can be viewed as a structural operator which means negation.
The basic structures are denoted by $\alpha, \beta, \gamma$ etc. with or without  subscripts. The {\em complexity} of a basic SDM-structure $\alpha$ is the number of all occurrences of $\lnot$, $\wedge$, $\vee$ and $*$ in $\alpha$.
An {\em SDM-structure} is a multi-set of basic structures.
Let $\Gamma, \Delta$ etc. with or without subscripts denote SDM-structures. 

Each SDM-structure is related with a term by the translation $t$ defined inductively as follows:
\[
t(\varphi) = \varphi, \quad
t(*\varphi) = \lnot\varphi,\quad
t(\Gamma_1, \Gamma_2) = t(\Gamma_1) \wedge t(\Gamma_2).
\]

An {\em SDM-sequent} is an expression of the form $\Gamma\Imp \alpha$ where $\Gamma$ is an SDM-structure and $\alpha$ is a basic SDM-structure. Each SDM-sequent $\Gamma\Imp\alpha$ is related with a basic sequent $t(\Gamma)\Imp t(\alpha)$.

\begin{definition}
The sequent calculus $\msf{G3SDM}$ for semi-De Morgan algebras consists of the following axioms and rules:
\begin{itemize}
\item Axioms: 
\[
(\mrm{Id})~p, \Gamma \Rightarrow p
\quad
(\bot\Imp)~\bot,\Gamma \Rightarrow \beta
\]
\[ 
(\Imp*\bot)~\Gamma \Rightarrow *\bot
\quad
(*\lnot\bot \Imp)~*\lnot\bot,\Gamma \Rightarrow \beta
\]
\item Logical rules:
\[
\AxiomC{$\varphi, \psi, \Gamma \Rightarrow \beta$}
\RightLabel{\small $(\wedge\Imp)$}
\UnaryInfC{$\varphi \wedge \psi, \Gamma \Rightarrow \beta$}
\DisplayProof
\quad
\AxiomC{$\Gamma \Rightarrow \varphi$}
\AxiomC{$\Gamma \Rightarrow \psi$}
\RightLabel{\small $(\Imp\wedge)$}
\BinaryInfC{$\Gamma \Rightarrow \varphi \wedge \psi$}
\DisplayProof
\]
\[
\AxiomC{$\varphi, \Gamma \Rightarrow \beta$}
\AxiomC{$\psi, \Gamma \Rightarrow \beta$}
\RightLabel{\small $(\vee\Imp)$}
\BinaryInfC{$\varphi \vee \psi, \Gamma \Rightarrow \beta$}
\DisplayProof
\quad
\AxiomC{$\Gamma \Rightarrow \varphi_{i}$}
\RightLabel{\small $(\Imp\vee) (i =1,2)$}
\UnaryInfC{$\Gamma \Rightarrow \varphi_{1} \vee \varphi_{2}$}
\DisplayProof
\]
\[
\AxiomC{$*\varphi, *\psi, \Gamma \Rightarrow \beta$}
\RightLabel{\small $(*\vee\Imp)$}
\UnaryInfC{$*(\varphi \vee \psi), \Gamma \Rightarrow \beta$}
\DisplayProof
\quad
\AxiomC{$\Gamma \Rightarrow *\varphi$}
\AxiomC{$\Gamma \Rightarrow *\psi$}
\RightLabel{\small $(\Imp*\vee)$}
\BinaryInfC{$\Gamma \Rightarrow *(\varphi\vee\psi)$}
\DisplayProof
\]
\[
\AxiomC{$*\lnot\varphi, *\lnot\psi, \Gamma \Rightarrow \beta$}
\RightLabel{\small $(*\lnot\wedge\Imp)$}
\UnaryInfC{$*\lnot(\varphi \wedge \psi), \Gamma \Rightarrow \beta$}
\DisplayProof
\quad
\AxiomC{$ \Gamma \Rightarrow *\lnot\varphi $}
\AxiomC{$ \Gamma \Rightarrow *\lnot\psi $}
\RightLabel{\small $(\Imp*\lnot\wedge)$}
\BinaryInfC{$ \Gamma \Rightarrow *\lnot(\varphi \wedge \psi) $}
\DisplayProof
\]
\[
\AxiomC{$*\varphi, \Gamma \Rightarrow \beta$}
\RightLabel{\small $(*\lnot\lnot\Imp)$}
\UnaryInfC{$*\lnot\lnot\varphi, \Gamma \Rightarrow \beta$}
\DisplayProof
\quad
\AxiomC{$ \Gamma \Rightarrow *\varphi $}
\RightLabel{\small $(\Imp*\lnot\lnot)$}
\UnaryInfC{$ \Gamma \Rightarrow *\lnot\lnot\varphi $}
\DisplayProof
\]
\[
\AxiomC{$*\varphi,\Gamma \Rightarrow \beta$}
\RightLabel{\small $(\lnot\Imp)$}
\UnaryInfC{$\lnot\varphi, \Gamma \Rightarrow \beta$}
\DisplayProof
\quad
\AxiomC{$ \Gamma \Rightarrow *\varphi $}
\RightLabel{\small $(\Imp\lnot)$}
\UnaryInfC{$ \Gamma \Rightarrow \lnot\varphi $}
\DisplayProof
\]
\item Structural rule:
\[
\AxiomC{$ \varphi \Rightarrow \psi $}
\RightLabel{\small $(*)$}
\UnaryInfC{$*\psi,\Gamma \Rightarrow *\varphi$}
\DisplayProof
\]

\end{itemize}
\end{definition}
A basic structure $\alpha$ is {\em principal} in an application of a rule $(\mathtt{R})$ if it is derived by $(\mathtt{R})$.
A {\em derivation} in $\msf{G3SDM}$ is either an axiom, or an application of a logical rule or structure rule to derivations concluding its premisses. The {\em height} of a derivation is the largest number of successive applications of rules, where an axiom has height $0$.
A sequent $\Gamma\Imp\alpha$ is {\em derivable} in $\msf{G3SDM}$, notation $\msf{G3SDM}\vdash \Gamma\Imp\alpha$, if it has a derivation in $\msf{G3SDM}$. 

The notation $\msf{G3SDM}\vdash_n\Gamma\Imp\alpha$ means that $\Gamma\Imp\alpha$ is derivable with a height of derivation at most $n$. 
A sequent rule $\frac{s_1 \ldots s_n}{s_0}$ (where $s_0,\ldots, s_n$ are SDM-sequents) is called {\em height-preserving admissible} in $\msf{G3SDM}$ if for any number $n$, $\msf{G3SDM}\vdash_n s_0$ whenever $\msf{G3SDM}\vdash_n s_i$ for all $1\leq i\leq n$.

\begin{theorem}
For any basic SDM-structures $\alpha$ and $\beta$, the weakening rule
\[
\AxiomC{$\Gamma \Rightarrow \beta$}
\RightLabel{\small $(Wk)$}
\UnaryInfC{$\alpha, \Gamma \Rightarrow \beta$}
\DisplayProof
\]
is height-preserving admissible in $\msf{G3SDM}$.
\end{theorem}

\begin{proof}
By induction on the height $n$ of derivation of $\Gamma \Rightarrow \beta$.
If $n = 0$, then $\Gamma \Rightarrow \beta$ is an axiom, and so is $ \alpha,\Gamma \Rightarrow \beta$. Let $n>0$.
Assume that $\vdash_{n} \Gamma \Rightarrow \beta$ is derived by $(\mathtt{R})$. 
If $(\mathtt{R})$ is a logical rule, the conclusion is obtained by applying $(\mathtt{R})$ to the induction hypothesis on premiss(es).
Assume $(\mathtt{R})=(*)$. Let $\Gamma = *\gamma,\Gamma'$ and $\beta = *\chi$. Then $\vdash_{n-1}\chi\Imp\gamma$. By $(*)$, $\vdash_{n}\alpha, *\gamma,\Gamma'\Imp *\chi$.
\end{proof}

\begin{lemma}\label{lemma:inversion}
For any number $n\geq 0$, terms $\varphi,\psi,\chi\in\mc{T}$ and SDM-structure $\Gamma$, the following hold in $\msf{G3SDM}$:
\begin{enumerate}
\item if $ \vdash_{n} \varphi \wedge \psi,\Gamma \Rightarrow \chi$, then $\vdash_{n} \varphi, \psi,\Gamma \Rightarrow \chi$.

\item if $\vdash_{n} \varphi \vee \psi,\Gamma \Rightarrow \chi$, then $\vdash_{n} \varphi, \Gamma \Rightarrow \chi$ and $\vdash_{n}  \psi, \Gamma\Rightarrow \chi$.

\item if $\vdash_{n}  *(\varphi \vee \psi), \Gamma \Rightarrow \chi$, then $\vdash_{n} *\varphi, *\psi, \Gamma \Rightarrow \chi$.

\item if $\vdash_{n} *\lnot(\varphi \wedge \psi),\Gamma  \Rightarrow \chi$, then $\vdash_{n}  *\lnot\varphi, *\lnot\psi, \Gamma \Rightarrow \chi$.

\item if $\vdash_{n} *\lnot\lnot\varphi, \Gamma  \Rightarrow \chi$, then $\vdash_{n} *\varphi, \Gamma \Rightarrow \chi$.

\item if $\vdash_{n} \lnot\varphi, \Gamma  \Rightarrow \chi$, then $\vdash_{n} *\varphi, \Gamma \Rightarrow \chi$.
\end{enumerate}
\end{lemma}
\begin{proof}
By induction on $n$. We sketch only the proof of (1). The case for $n=0$ is obvious. Let $n>0$. For the inductive step, assume $\vdash_{n} \varphi \wedge \psi, \Gamma \Rightarrow \chi$ and the last rule is $(\mathtt{R})$. If $\varphi \wedge \psi$ is principal in $(\mathtt{R})$, then $\vdash_{n-1}\varphi, \psi, \Gamma \Rightarrow \chi$.
If $\varphi \wedge \psi$ is not principal in $(\mathtt{R})$, the conclusion is obtained by applying $(\mathtt{R})$ to the induction hypothesis on the premiss(es) of $(\mathtt{R})$. Note that $(\mathtt{R})$ cannot be $(*)$ because the succedent of the sequent is a term.
\end{proof}

\begin{theorem}
For any basic SDM-structure $\alpha$ and term $\psi\in\mc{T}$, the contraction rule
\[
\AxiomC{$\alpha, \alpha,\Gamma \Rightarrow \psi$}
\RightLabel{\small $(Ctr)$}
\UnaryInfC{$\alpha,\Gamma \Rightarrow \psi$}
\DisplayProof
\]
is height-preserving admissible in $\msf{G3SDM}$.
\end{theorem}

\begin{proof}
By induction on the height $n$ of derivation of $\alpha,\alpha, \Gamma \Rightarrow \psi$. The case $n=0$ is obvious. Let $n>0$. Assume that $\vdash_{n} \alpha,\alpha,\Gamma \Rightarrow \psi$ and the last rule is $(\mathtt{R})$. If $\alpha$ is not principal in $(\mathtt{R})$, the conclusion is obtained by applying $(\mathtt{R})$ to the induction hypothesis on 
the premiss(es).
Assume that $\alpha$ is principle in $(\mathtt{R})$.
We have the following cases:

Case 1. $\alpha=\varphi_{1} \wedge \varphi_{2}$. Let the premiss of $(\mathtt{R})$ be $\vdash_{n-1} \varphi_{1}, \varphi_{2}, \varphi_{1} \wedge \varphi_{2}, \Gamma \Rightarrow \psi$. 
By Lemma \ref{lemma:inversion} (1), $\vdash_{n-1} \varphi_{1}, \varphi_{2}, \varphi_{1}, \varphi_{2}, \Gamma \Rightarrow \psi$. By induction hypothesis applied twice, $\vdash_{n-1} \varphi_{1}, \varphi_{2}, \Gamma \Rightarrow \psi$. By $(\wedge\Imp)$, $\vdash_{n} \varphi_{1} \wedge \varphi_{2}, \Gamma \Rightarrow \psi$.

Case 2. $\alpha=\varphi_{1} \vee \varphi_{2}$. The premisses of $(\mathtt{R})$ are $ \vdash_{n-1} \varphi_{1}, \varphi_{1} \vee \varphi_{2},\Gamma \Rightarrow \psi$ and $ \vdash_{n-1} \varphi_{2}, \varphi_{1} \vee \varphi_{2}, \Gamma \Rightarrow \psi$. 
By Lemma \ref{lemma:inversion} (2), $ \vdash_{n-1} \varphi_{1}, \varphi_{1}, \Gamma \Rightarrow \psi$ and $ \vdash_{n-1} \varphi_{2}, \varphi_{2}, \Gamma\Rightarrow \psi$. By induction hypothesis, $ \vdash_{n-1} \varphi_{1}, \Gamma \Rightarrow \psi$ and $ \vdash_{n-1} \varphi_{2}, \Gamma \Rightarrow \psi$. By $(\vee\Imp)$, $ \vdash_{n} \varphi_{1} \vee \varphi_{2}, \Gamma \Rightarrow \psi$.

Case 3. $\alpha = *(\varphi_{1} \vee \varphi_{2})$. 
The premisses of $(\mathtt{R})$ is $ \vdash_{n-1} *\varphi_{1}, *\varphi_{2}, *(\varphi_{1} \vee \varphi_{2}), \Gamma\Rightarrow \psi$. By Lemma \ref{lemma:inversion} (3), $ \vdash_{n-1} *\varphi_{1}, *\varphi_{2}, *\varphi_{1}, *\varphi_{2}, \Gamma \Rightarrow \psi$. By induction hypothesis, $ \vdash_{n-1} *\varphi_{1}, *\varphi_{2}, \Gamma \Rightarrow \psi$. By $(*\vee\Imp)$, $ \vdash_{n} *(\varphi_{1} \vee \varphi_{2}), \Gamma \Rightarrow \psi$.

Case 4. $\alpha = *\lnot(\varphi_{1} \wedge \varphi_{2})$. The premiss of $(\mathtt{R})$ is $ \vdash_{n-1} *\lnot\varphi_{1}, *\lnot\varphi_{2},$ $*\lnot(\varphi_{1} \vee \varphi_{2}), \Gamma \Rightarrow \psi$. By Lemma \ref{lemma:inversion} (4), $ \vdash_{n-1} *\lnot\varphi_{1}, *\lnot\varphi_{2}, *\lnot\varphi_{1}, *\lnot\varphi_{2},$ $\Gamma \Rightarrow \psi$. By induction hypothesis, $ \vdash_{n-1} *\lnot\varphi_{1}, *\lnot\varphi_{2}, \Gamma\Rightarrow \psi$. By $(*\lnot\wedge\Imp)$, $ \vdash_{n} *\lnot(\varphi_{1} \wedge \varphi_{2}), \Gamma \Rightarrow \psi$.

Case 5. $\alpha = *\lnot\lnot\varphi$.  The premiss of $(\mathtt{R})$ is $ \vdash_{n-1} *\varphi, *\lnot\lnot\varphi, \Gamma \Rightarrow \psi$. By Lemma \ref{lemma:inversion} (5), $ \vdash_{n-1} *\varphi, *\varphi,\Gamma\Rightarrow \psi$. By induction hypothesis, $ \vdash_{n-1} *\varphi, \Gamma\Rightarrow \psi$. By $(*\lnot\lnot\Imp)$, $ \vdash_{n} *\lnot\lnot\varphi, \Gamma \Rightarrow \psi$.

Case 6. $\alpha = \lnot\varphi_{1}$. The premiss of $(\mathtt{R})$ is $ \vdash_{n-1} *\varphi_{1}, \lnot\varphi_{1}, \Gamma \Rightarrow \psi$. By Lemma \ref{lemma:inversion} (6), $ \vdash_{n-1} *\varphi_{1}, *\varphi_{1}, \Gamma\Rightarrow \psi$. By induction hypothesis, $ \vdash_{n-1} *\varphi_{1},\Gamma\Rightarrow \psi$. By $(\lnot\Imp)$, $ \vdash_{n} \lnot\varphi_{1}, \Gamma \Rightarrow \psi$.
\end{proof}

\begin{theorem}\label{thm:cut_sdm}
For any basic SDM-structure $\alpha$, term $\psi\in\mc{T}$, and SDM-structures $\Gamma$ and $\Delta$, the restricted cut rule
\[
\AxiomC{$ \Gamma \Rightarrow \alpha $\quad$ \alpha, \Delta \Rightarrow \psi$}
\RightLabel{\small $(Cut^{*})$}
\UnaryInfC{$ \Gamma, \Delta \Rightarrow \psi $}
\DisplayProof
\]
is admissible in $\msf{G3SDM}$.
\end{theorem}

\begin{proof}
By simultaneous induction on (i) the height $m$ of derivation of the left premiss $ \Gamma \Rightarrow \alpha $; (ii) the height $n$ of derivation of the right premiss $\alpha,\Delta\Rightarrow \psi $; and (iii) the complexity of $\alpha$.
Let $\Gamma\Imp\alpha$ be obtained by $(\mathtt{R}_1)$ and $\alpha, \Delta\Imp\psi$ be obtained by $(\mathtt{R}_2)$. Note that $(\mathtt{R}_2)$ can not be the rule $(*)$.

Assume that one of $(\mathtt{R}_1)$ and $(\mathtt{R}_2)$ is an axiom. We have two cases:

Case 1. $(\mathtt{R}_1)$ is an axiom. We have the following cases:

(1.1) $(\mathtt{R}_1)$ is ($Id$). Then we obtain $\Gamma,\Delta \Rightarrow \psi$ from $\alpha, \Delta  \Rightarrow \psi$ by $(Wk)$.

(1.2) $(\mathtt{R}_1)$ is $(\bot\Imp)$. Then $\Gamma, \Delta \Rightarrow \psi$ is an instance of $(\bot\Imp)$. 

(1.3) $(\mathtt{R}_1)$ is $(\Imp*\bot)$. Then $\alpha = *\bot$. If $(\mathtt{R}_2)$ is an instance of an axiom, then the conclusion $\Gamma,\Delta\Imp\psi$ is an instance of an axiom. Suppose that $(\mathtt{R}_2)$ is not an axiom. Clearly $\alpha$ is not principal in $(\mathtt{R}_2)$. By induction hypothesis (ii), we apply $(Cut^{*})$ to the premiss(es) of $(\mathtt{R}_2)$ and then apply $(\mathtt{R}_2)$.

Case 2. $(\mathtt{R}_2)$ is an axiom. We have the following cases:

(2.1) $(\mathtt{R}_2)$ is ($Id$). Then $\psi=p$ for some variable $p$, and $p$ is in $\Delta$ or $\alpha=p$. If $\psi=p$ is in $\Delta$, then $\Gamma, \Delta\Imp \psi$ is an instance of ($Id$). If $\alpha=p=\psi$, then $\Gamma,\Delta\Imp\psi$ is obtained from $\Gamma\Imp\alpha$ by $(Wk)$.

(2.2) $(\mathtt{R}_2)$ is one of $(\bot\Imp)$, $(*\lnot\bot\Imp)$. If $\alpha\not\in\{\bot, *\lnot\bot\}$,  then $\Gamma, \Delta\Imp\psi$ is obtained by $(\mathtt{R}_2)$.
Suppose $\alpha\in\{\bot, *\lnot\bot\}$. If $(\mathtt{R}_1)$ is an instance of an axiom, $\Gamma,\Delta\Imp\psi$ is also an instance of an axiom. Suppose that $(\mathtt{R}_1)$ is not an axiom. Clearly $\alpha$ is not principal in $(\mathtt{R}_1)$. By induction hypothesis (i), we apply $(Cut^{*})$ to the premiss(es) of $(\mathtt{R}_1)$ and then apply $(\mathtt{R}_1)$.

Assume that neither $(\mathtt{R}_1)$ nor $(\mathtt{R}_2)$ is an axiom. We have three cases:

Case 3. $\alpha$ is not principal in $(\mathtt{R}_1)$. One can prove the admissibility of $(Cut^{*})$ by induction on $m$. For example, the derivation
\[
\AxiomC{$\varphi_{1}, \varphi_{2}, \Gamma' \Rightarrow \alpha$}
\RightLabel{\small $(\wedge\Imp)$}
\UnaryInfC{$ \varphi_{1} \wedge \varphi_{2}, \Gamma' \Rightarrow \alpha$}
\AxiomC{$\alpha, \Delta \Rightarrow \psi $}
\RightLabel{\small $(Cut^{*})$}
\BinaryInfC{$\varphi_{1} \wedge \varphi_{2},\Gamma, \Delta \Rightarrow \psi$}
\DisplayProof
\]
is transformed into
\[
\AxiomC{$\varphi_{1},  \varphi_{2}, \Gamma' \Rightarrow \alpha$}
\AxiomC{$ \alpha, \Delta \Rightarrow \psi  $}
\RightLabel{\small $(Cut^{*})$}
\BinaryInfC{$\varphi_{1},  \varphi_{2},  \Gamma', \Delta \Rightarrow \psi$}
\RightLabel{\small $(\wedge\Imp)$}
\UnaryInfC{$\varphi_{1} \wedge \varphi_{2}, \Gamma', \Delta \Rightarrow \psi$}
\DisplayProof
\]

Case 4. $\alpha$ is principal only in $(\mathtt{R}_1)$. Then the admissibility of $(Cut^{*})$ can be proved by induction on $n$. For example, the derivation
\[
\AxiomC{$ \Gamma \Rightarrow \alpha $}
\AxiomC{$\alpha, *\chi, \Delta' \Rightarrow \psi$}
\RightLabel{\small $(\lnot\Imp)$}
\UnaryInfC{$ \alpha, \lnot\chi, \Delta' \Rightarrow \psi$}
\RightLabel{\small $(Cut^{*})$}
\BinaryInfC{$ \lnot\chi, \Gamma, \Delta' \Rightarrow \psi$ }
\DisplayProof
\] 
is transformed into
\[
\AxiomC{$ \Gamma \Rightarrow \alpha $}
\AxiomC{$\alpha, *\chi, \Delta' \Rightarrow \psi$}
\RightLabel{\small $(Cut^{*})$}
\BinaryInfC{$*\chi, \Gamma, \Delta' \Rightarrow \psi$ }
\RightLabel{\small $(\lnot\Imp)$}
\UnaryInfC{$ \lnot\chi, \Gamma, \Delta' \Rightarrow \psi$ }
\DisplayProof
\]

Case 5. $\alpha$ is principal in both premisses. Then we have the following cases according to the complexity of $\alpha$.

(5.1) $\alpha = \varphi_{1} \wedge \varphi_{2} $, $ \alpha = \varphi_{1} \vee \varphi_{2} $, $ \alpha = *(\varphi_{1} \vee \varphi_{2} )$, or $ \alpha = *\lnot(\varphi_{1} \wedge \varphi_{2} )$. It is quite easy to push up the $(Cut^{*})$ to premiss(es) where the cut basic structure is less complex. For example, the derivation
\[
\AxiomC{$\Gamma \Rightarrow \varphi_{1}$}
\AxiomC{$\Gamma \Rightarrow \varphi_{2}$}
\RightLabel{\small $(\wedge\Imp)$}
\BinaryInfC{$\Gamma \Rightarrow \varphi_{1} \wedge \varphi_{2} $}
\AxiomC{$\varphi_{1}, \varphi_{2}, \Delta \Rightarrow \psi$}
\RightLabel{\small $(\wedge\Imp)$}
\UnaryInfC{$ \varphi_{1} \wedge \varphi_{2}, \Delta \Rightarrow \psi$}
\RightLabel{\small $(Cut^{*})$}
\BinaryInfC{$\Gamma, \Delta \Rightarrow \psi$}
\DisplayProof
\]
is transformed into
\[
\AxiomC{$\Gamma \Rightarrow \varphi_{2}$}
\AxiomC{$\Gamma \Rightarrow \varphi_{1}$}
\AxiomC{$\varphi_1, \varphi_2, \Delta \Rightarrow \psi$}
\RightLabel{\small $(Cut^{*})$}
\BinaryInfC{$\Gamma, \varphi_{2}, \Delta \Rightarrow \psi$}
\RightLabel{\small $(Cut^{*})$}
\BinaryInfC{$ \Gamma, \Gamma, \Delta\Rightarrow \psi$}
\RightLabel{\small $(ctr)$}
\UnaryInfC{$ \Gamma, \Delta \Rightarrow \psi$}
\DisplayProof
\]

(5.2) $\alpha= *\lnot\lnot\varphi$. The application of $(Cut^{*})$ is push up to premiss(es) where $(Cut^{*})$ is applied to sequents of less height. The derivation
\[
\AxiomC{$\Gamma \Rightarrow *\varphi$}
\RightLabel{\small $(\Imp *\lnot\lnot)$}
\UnaryInfC{$\Gamma \Rightarrow *\lnot\lnot\varphi$}
\AxiomC{$*\varphi, \Delta \Rightarrow \psi$}
\RightLabel{\small $(*\lnot\lnot\Imp)$}
\UnaryInfC{$*\lnot\lnot\varphi, \Delta \Rightarrow \psi$}
\RightLabel{\small $(Cut^{*})$}
\BinaryInfC{$\Gamma, \Delta \Rightarrow \psi$}
\DisplayProof
\]
is transformed into
\[
\AxiomC{$\Gamma \Rightarrow *\varphi$}
\AxiomC{$*\varphi, \Delta \Rightarrow \psi$}
\RightLabel{\small $(Cut^{*})$}
\BinaryInfC{$\Gamma, \Delta \Rightarrow \psi$}
\DisplayProof
\]

(5.3) $\alpha= \lnot\varphi$. The application of $(Cut^*)$ is push up to premiss(es) where $(Cut^{*})$ is applied to sequents of less height. The derivation
\[
\AxiomC{$\Gamma \Rightarrow *\varphi$}
\RightLabel{\small $(\Imp\lnot)$}
\UnaryInfC{$\Gamma \Rightarrow \lnot\varphi$}
\AxiomC{$*\varphi, \Delta \Rightarrow \psi$}
\RightLabel{\small $(\lnot\Imp)$}
\UnaryInfC{$\lnot\varphi, \Delta \Rightarrow \psi$}
\RightLabel{\small $(Cut^{*})$}
\BinaryInfC{$\Gamma, \Delta \Rightarrow \psi$}
\DisplayProof
\]
is transformed into
\[
\AxiomC{$\Gamma \Rightarrow *\varphi$}
\AxiomC{$*\varphi, \Delta \Rightarrow \psi$}
\RightLabel{\small $(Cut^{*})$}
\BinaryInfC{$\Gamma, \Delta \Rightarrow \psi$}
\DisplayProof
\]
This completes the proof.
\end{proof}

\begin{lemma}\label{exchange}
(1) $\msf{G3SDM} \vdash \Gamma \Imp \lnot\varphi$ if and only if $\msf{G3SDM} \vdash \Gamma \Imp *\varphi$; (2) $\msf{G3SDM} \vdash \lnot\varphi, \Gamma \Imp \psi$ if and only if $\msf{G3SDM} \vdash *\varphi, \Gamma \Imp \psi$.
\end{lemma}
\begin{proof}
For (1), the `if' part is obtained by $(\Imp\lnot)$. The `only if' part is easily shown by induction on the height $n$ of derivation of $\Gamma \Imp \lnot\varphi$ in $\msf{G3SDM}$. For (2), the `if' part is obtained by $(\lnot\Imp)$. The `only if' part
is obtained by Lemma \ref{lemma:inversion} (6).
\end{proof}

\begin{theorem}\label{thm:cut_general}
For any basic SDM-structures $\alpha$ and $\beta$, the full cut rule
\[
\AxiomC{$ \Gamma \Rightarrow \alpha $\quad$ \alpha, \Delta \Rightarrow \beta$}
\RightLabel{\small $(Cut)$}
\UnaryInfC{$ \Gamma, \Delta \Rightarrow \beta$}
\DisplayProof
\]
is admissible in $\msf{G3SDM}$.
\end{theorem}

\begin{proof}
Assume that $\msf{G3SDM}\vdash \Gamma \Rightarrow \alpha $ and $\msf{G3SDM}\vdash \alpha, \Delta \Rightarrow \beta$.
By Lemma \ref{exchange} (1), $\msf{G3SDM}\vdash \alpha, \Delta \Rightarrow t(\beta)$. By $(Cut^*)$, we obtain $\msf{G3SDM}\vdash \Gamma, \Delta \Rightarrow t(\beta)$. By Lemma \ref{exchange} (1), $\msf{G3SDM}\vdash \Gamma, \Delta \Rightarrow \beta$.
\end{proof}


\subsection{Completeness}
In this subsection, we shall show the completeness of $\msf{G3SDM}$ with respect to the variety $\msf{SDM}$.
A sequent $\Gamma \Imp \alpha$ is {\em valid} in an algebra $\mf{A}$, notation $\mf{A}\models\Gamma \Imp \alpha$, if $\mf{A}\models t(\Gamma) \Imp t(\alpha)$. 
The notation $\msf{SDM}\models\Gamma\Imp \alpha$ means that $\Gamma\Imp\alpha$ is valid in all semi-De Morgan algebras.

\begin{proposition}\label{prop:ctp}
The contraposition rule
\[
\AxiomC{$\varphi \Rightarrow \psi$}
\RightLabel{\small $(\mrm{CP})$}
\UnaryInfC{$\lnot\psi,\Gamma \Rightarrow \lnot\varphi$}
\DisplayProof
\]
is admissible in $\mathsf{G3SDM}$.
\end{proposition}
\begin{proof}
By $(*)$, $(\lnot\Imp)$ and $(\Imp\lnot)$.
\end{proof}

\begin{lemma}\label{lemma:id}
(1) $\msf{G3SDM}\vdash\varphi, \Gamma \Rightarrow \varphi$ and (2) $\msf{G3SDM} \vdash *\varphi, \Gamma \Rightarrow *\varphi$.
\end{lemma}
\begin{proof}
(2) follows from (1) by the rule $(*)$.
(1) is shown by induction on the complexity of $\varphi$.
The atomic case is obvious. The cases for conjunction and disjunction are easy. Let $\varphi = \lnot\psi$. By induction hypothesis, $\msf{GSDM} \vdash \psi \Rightarrow \psi$. By (CP), $\lnot\psi, \Gamma \Rightarrow \lnot\psi$.
\end{proof}

\begin{lemma}\label{lemma:t1}
If $S_\msf{SDM} \vdash \varphi \Rightarrow \psi$, then $\msf{G3SDM} \vdash \varphi \Rightarrow \psi$.
\end{lemma}

\begin{proof}
By induction on the derivation of $\varphi \Rightarrow \psi$ in $S_\msf{SDM}$.
Note that the axiom $(Id)$ is obtained by Lemma \ref{lemma:id}. The contraposition rule ($Ctp$) in $S_\msf{SDM}$ is a special case of (CP) in Proposition \ref{prop:ctp}. 
\end{proof}

\begin{lemma}\label{lemma:t2}
If $\msf{G3SDM} \vdash t(\Gamma) \Rightarrow t(\alpha)$, then $\msf{G3SDM} \vdash \Gamma \Rightarrow \alpha$.
\end{lemma}
\begin{proof}
Assume $\msf{G3SDM} \vdash t(\Gamma) \Rightarrow t(\alpha)$.
Let $\Gamma = \alpha_{1}, \ldots, \alpha_n$, where each $\alpha_{i}$ is a basic structure. Then $t(\Gamma) = t(\alpha_1)\wedge \ldots \wedge t(\alpha_{n})$. By Lemma \ref{lemma:inversion} (1), $\msf{G3SDM} \vdash t(\alpha_{1}), \ldots, t(\alpha_{n}) \Rightarrow t(\alpha)$. By Lemma \ref{exchange}, $\msf{G3SDM} \vdash \Gamma \Rightarrow \alpha$.
\end{proof}

\begin{lemma}\label{lemma:t3}
If $S_\msf{SDM} \vdash t(\Gamma) \Rightarrow t(\alpha)$, then $\msf{G3SDM} \vdash \Gamma \Rightarrow \alpha$.
\end{lemma}
\begin{proof}
By Lemma \ref{lemma:t1} and Lemma \ref{lemma:t2}.
\end{proof}

\begin{theorem}
$\msf{G3SDM} \vdash \Gamma \Rightarrow \alpha$ if and only if 
$\msf{SDM} \models \Gamma \Rightarrow \alpha$.
\end{theorem}
\begin{proof}
The soundness is easily shown by induction on the height of derivation of $\Gamma\Imp\alpha$. For the completeness, assume $\msf{G3SDM} \not\vdash \Gamma \Rightarrow \psi$. By Lemma \ref{lemma:t3}, $S_\msf{SDM}\not\vdash t(\Gamma)\Imp t(\alpha)$. By the completeness of $S_\msf{SDM}$, there is a semi-De Morgan algebra $\mf{A}$ such that $\mf{A}\not\models t(\Gamma)\Imp t(\alpha)$. Hence $\mf{A}\not\models \Gamma \Imp \alpha $.
\end{proof}

\subsection{Craig interpolation and decidability}
For any basic SDM-structure $\alpha$, let $var(\alpha)$ be the set of all propositional variables occurred in $\alpha$. For any SDM-structure $\Gamma = ( \alpha_1,\ldots, \alpha_n )$, let 
$var(\Gamma) = var(\alpha_1) \cup \ldots \cup var(\alpha_n)$.

\begin{definition}
Given any SDM-sequent $\Gamma\Imp \beta$, we say that $(\Gamma_1; \emptyset)(\Gamma_2; \beta)$ is a {\em partition} of $\Gamma\Imp\beta$, if the multiset union of $\Gamma_1$ and $\Gamma_2$ is equal to $\Gamma$.  
Let $\msf{G3SDM} \vdash \Gamma \Rightarrow \beta$. A basic SDM-structure $\alpha$ is called an {\em interpolant} of the partition $(\Gamma_1; \emptyset)(\Gamma_2; \beta)$ if (i) $\msf{G3SDM} \vdash  \Gamma_1 \Rightarrow \alpha$, (ii) $\msf{G3SDM} \vdash \alpha, \Gamma_2 \Rightarrow \beta$ and (iii) $var(\alpha)\sub var(\Gamma_1)\cap var(\Gamma_2, \beta)$.
\end{definition}

Let $\alpha$ be an interpolant of the partition $(\Gamma_1; \emptyset)(\Gamma_2; \beta)$. By Lemma \ref{exchange}, it is obvious that the term $t(\alpha)$ is also an interpolant of the partition.

\begin{theorem}[Interpolation]\label{thm:int_sdm}
For any SDM-sequent $\Gamma\Imp\beta$, if $\msf{G3SDM} \vdash \Gamma \Rightarrow \beta$, then any partition of $\Gamma\Imp\beta$ has an interpolant.
\end{theorem}

\begin{proof}
By induction on the height $n$ of derivation of $\Gamma \Rightarrow \beta$ in $\msf{G3SDM}$. For the case $n=0$, 
we show only the case that $\Gamma\Imp\beta$ is an instance of $(Id)$, and the remaining cases are easy.
Let $\beta = p$ and $\Gamma = \Gamma_1,\Gamma_2, p$.
For the partition $(\Gamma_1, p;\emptyset)(\Gamma_2; \beta)$ of $\Gamma$, choose $p$ as an interpolant.
For the partition $(\Gamma_1;\emptyset)(\Gamma_2, p;\beta)$, choose $*\bot$ as an interpolant.

Assume that $ \vdash_{n+1}\Gamma \Rightarrow \beta$ and the last rule is $(\mathtt{R})$. 
If $(\mathtt{R})$ is a right logical rule, then the interpolant is obtained by induction hypothesis and the rule $(\mathtt{R})$. 
For example, $(\mathtt{R})$ is $(\Imp\wedge)$. Let $\beta = \varphi_{1} \wedge \varphi_2$.
Then the premisses of $(\mathtt{R})$ are $\Gamma \Imp\varphi_1$ and $\Gamma \Imp\varphi_2$, and the conclusion is 
$\Gamma \Imp \varphi_1\wedge\varphi_2$. 
Let $(\Gamma_1;\emptyset)(\Gamma_2; \beta)$ be any partition of $\Gamma\Imp\beta$. 
By induction hypothesis, there are interpolants $\alpha_1$ and $\alpha_2$ such that: 
(i) $\vdash \Gamma_1 \Rightarrow \alpha_1$, 
(ii) $ \vdash \alpha_{1}, \Gamma_2 \Rightarrow \varphi_1$, 
(iii) $var(\alpha_1)\sub var(\Gamma_1) \cap var(\Gamma_2, \varphi_1)$, 
(iv) $\vdash \Gamma_1 \Rightarrow \alpha_2$, 
(v) $ \vdash \alpha_{2}, \Gamma_2 \Rightarrow \varphi_2$, 
(vi) $var(\alpha_2)\sub var(\Gamma_1)\cap var(\Gamma_2, \varphi_2)$.
By (i) and (iv), using $(\Imp\wedge)$, we get
$\vdash \Gamma_1 \Rightarrow \alpha_1 \wedge \alpha_2$. 
By (ii) and (v), using $(Wk)$, $(\wedge\Imp)$ and $(\Imp\wedge)$, we get
$ \vdash \alpha_1 \wedge \alpha_2, \Gamma_2 \Rightarrow \varphi_{1} \wedge \varphi_2$.
By (iii) and (vi), $var(\alpha_1 \wedge \alpha_2)\sub var(\Gamma_1) \cap var(\Gamma_2, \varphi_1 \wedge \varphi_2)$.
Then $\alpha_1\wedge\alpha_2$ is a required interpolant. 

Suppose that $(\mathtt{R})$ is a left logical rule. The proof is done by induction hypothesis. For example, 
$(\mathtt{R})$ is $(\wedge\Imp)$. Let $\Gamma = \varphi_{1} \wedge \varphi_2, \Gamma_1', \Gamma_2'$. Then the premiss of $(\mathtt{R})$ is $\varphi_1,\varphi_2, \Gamma_1',\Gamma_2'\Imp\beta$, and the conclusion is 
$\varphi_{1} \wedge \varphi_2, \Gamma_1', \Gamma_2'\Imp \beta$.
Consider the partition $(\varphi_{1} \wedge \varphi_2, \Gamma_1^{'}; \Gamma_2')$. By induction hypothesis, there is an interpolant $\alpha$ such that (i) $\msf{G3SDM} \vdash \varphi_{1},\varphi_2, \Gamma_1^{'} \Rightarrow \alpha$, (ii) $\msf{G3SDM} \vdash \alpha, \Gamma_2' \Rightarrow \beta$ and (iii) $var(\alpha)\sub var(\varphi_{1}, \varphi_2, \Gamma_1^{'})\cap var(\Gamma_2', \beta)$. Then apply $(\wedge\Imp)$ to (i), we get $ \vdash_{n+1} \varphi_{1} \wedge \varphi_2, \Gamma_1^{'} \Rightarrow \alpha$. By (iii), $var(\alpha)\sub var(\varphi_{1} \wedge \varphi_2, \Gamma_1^{'})\cap var(\Gamma_2', \beta)$. Then $\alpha$ is an interpolant for the partition $(\varphi_{1} \wedge \varphi_2, \Gamma_1^{'}; \Gamma_2')$.
For the partition $(\Gamma_1^{'}; \varphi_{1} \wedge \varphi_2, \Gamma_2')$, the argument is similar.
Suppose that $(\mathtt{R})$ is $(*)$. 
Then the premiss of $(\mathtt{R})$ is 
$\varphi \Imp\psi$ and the conclusion is
$*\psi\Imp *\varphi$. For the partition $(\emptyset;\emptyset)(*\psi;*\varphi)$, choose $*\bot$ as an interpolant. For the partition $(*\psi;\emptyset)(\emptyset;*\varphi)$, choose $*\psi$ as an interpolant.
\end{proof}

Now we shall show the decidability of derivability in $\msf{G3SDM}$. This system has no standard subformula property. But one can show that the proof search for a sequent is bounded in some number.

\begin {definition}
The {\em SDM-weight} $w(\varphi)$ of a term $\varphi$ is defined as follows:
\begin{align*}
w(p) &= w(\bot) = 1 & 
w(\lnot\varphi) &= w(\varphi) + 2\\
w(\varphi \vee \psi) &= w(\varphi) + w(\psi) + 2 &
w(\varphi \wedge \psi) &= w(\varphi) + w(\psi) + 3
\end{align*}

The weight of a basic structure $*\varphi$ is defined as $w(*\varphi) = w(\varphi)+1$. Given an SDM-structure $\Gamma = (\alpha_1,\ldots,\alpha_n)$, the weight of $\Gamma$ is defined as $w(\Gamma) = w(\alpha_1)+\ldots+w(\alpha_n)$.
\end{definition}

\begin{lemma}\label{dec:less weight}
In any logical rule or structural rule in $\msf{G3SDM}$, the weight of each premiss is strictly less than the weight of the conclusion.
\end{lemma}
\begin{proof}
By straightforward inspection on all rules in $\msf{G3SDM}$.
\end{proof}

The caluclus $\msf{G3SDM}$ is contraction-free and cut-free. Using Lemma \ref{dec:less weight},  we obtain the following decidability result by proof search.

\begin{theorem}[Decidability]\label{dec:sdm}
The derivability of an SDM-sequent in the calculus $\msf{G3SDM}$ is decidable.
\end{theorem}


\section{A Sequent Calculus for De Morgan Algebras}
In this section, we shall present a single-conclusion sequent calculus $\msf{G3DM}$ for De Morgan algebras. Then we shall show the Craig interpolation and the decidability of $\msf{G3DM}$.

\subsection{The sequent calculus $\msf{G3DM}$}
A {\em DM-structure} is a finite multiset of terms. All DM-structures are 
denoted by $\Sigma, \Theta$ etc. with or without subscripts.
A {\em DM-sequent} is an expression $\Sigma\Imp\varphi$ where $\Sigma$ is a DM-structure and $\varphi$ is a term.

\begin{definition}
The sequent calculus $\msf{G3DM}$ for De Morgan algebras consists of the following axioms and rules:
\begin{itemize}
\item Axioms: 
\[
({Id}_1)~p, \Sigma \Rightarrow p
\quad
({Id}_2)~\lnot p, \Sigma \Rightarrow \lnot p
\]
\[
(\bot\Imp)~\bot,\Sigma \Rightarrow \varphi
\quad
(\Imp\lnot \bot )~\Sigma \Rightarrow \lnot \bot
\]
\item Logical rules:
\[
\AxiomC{$\varphi, \psi, \Sigma \Rightarrow \chi$}
\RightLabel{\small $(\wedge\Imp)$}
\UnaryInfC{$\varphi \wedge \psi, \Sigma \Rightarrow \chi$}
\DisplayProof
\quad
\AxiomC{$\Sigma \Rightarrow \varphi$\quad$\Sigma \Rightarrow \psi$}
\RightLabel{\small $(\Imp\wedge)$}
\UnaryInfC{$\Sigma \Rightarrow \varphi \wedge \psi$}
\DisplayProof
\]
\[
\AxiomC{$\varphi, \Sigma \Rightarrow \chi$}
\AxiomC{$\psi, \Sigma \Rightarrow \chi$}
\RightLabel{\small $(\vee\Imp)$}
\BinaryInfC{$\varphi \vee \psi, \Sigma \Rightarrow \chi$}
\DisplayProof
\quad
\AxiomC{$\Sigma \Rightarrow \varphi_{i}$}
\RightLabel{\small $(\Imp\vee)(i=1,2)$}
\UnaryInfC{$\Sigma \Rightarrow \varphi_{1} \vee \varphi_{2}$}
\DisplayProof
\]
\[
\AxiomC{$\lnot\varphi, \Sigma \Rightarrow \chi$\quad$\lnot\psi, \Sigma \Rightarrow \chi$}
\RightLabel{\small $(\lnot\wedge\Imp)$}
\UnaryInfC{$\lnot(\varphi \wedge \psi), \Sigma \Rightarrow \chi$}
\DisplayProof
\quad
\AxiomC{$\Sigma \Rightarrow \lnot\varphi_{i}$}
\RightLabel{\small $(\Imp\lnot\wedge)(i=1,2)$}
\UnaryInfC{$\Sigma \Rightarrow \lnot(\varphi_2 \wedge \varphi_2)$}
\DisplayProof
\]
\[
\AxiomC{$\lnot\varphi, \lnot\psi, \Sigma \Rightarrow \chi$}
\RightLabel{\small $(\lnot\vee\Imp)$}
\UnaryInfC{$\lnot(\varphi \vee \psi), \Sigma \Rightarrow \chi$}
\DisplayProof
\quad
\AxiomC{$\Sigma \Rightarrow \lnot\varphi$\quad$\Sigma \Rightarrow \lnot\psi$}
\RightLabel{\small $(\Imp\lnot\vee)$}
\UnaryInfC{$\Sigma \Rightarrow \lnot(\varphi\vee\psi)$}
\DisplayProof
\]
\[
\AxiomC{$\varphi, \Sigma \Rightarrow \chi$}
\RightLabel{\small $(\lnot\lnot\Imp)$}
\UnaryInfC{$\lnot\lnot\varphi, \Sigma \Rightarrow \chi$}
\DisplayProof
\quad
\AxiomC{$ \Sigma \Rightarrow \varphi $}
\RightLabel{\small $(\Imp\lnot\lnot)$}
\UnaryInfC{$ \Sigma \Rightarrow \lnot\lnot\varphi $}
\DisplayProof
\]
\end{itemize}
~\\
The notation $\msf{G3DM}\vdash\Sigma\Imp\varphi$ stands for that $\Sigma\Imp\varphi$ is derivable in $\msf{G3DM}$.
\end{definition}

\begin{theorem}
The weakening rule
\[
\AxiomC{$ \Sigma \Rightarrow \varphi$}
\RightLabel{\small $(Wk)$}
\UnaryInfC{$\psi, \Sigma \Rightarrow \varphi$}
\DisplayProof
\]
is height-preserving admissible in $\msf{G3DM}$.
\end{theorem}
\begin{proof}
By induction on the height $n$ of derivation of $\Sigma\Imp \varphi$ in $\msf{G3DM}$. When $n = 0$, $\Sigma \Rightarrow \varphi$ is an axiom, and obviously $ \vdash_0\psi,\Sigma \Rightarrow \varphi$.
Assume that $\vdash_{n + 1} \Sigma \Rightarrow \varphi$ and the last rule is $(\mathtt{R})$. We apply $(Wk)$ to the premiss(es) of $(\mathtt{R})$ and then apply $(\mathtt{R})$. For example, 
$(\mathtt{R})$ is the rule $(\lnot\vee\Imp)$. Let $\Sigma = \lnot(\psi_{1} \vee \psi_{2}), \Sigma^{'}$.
Then the premiss of $(\mathtt{R})$ is $\lnot\psi_{1}, \lnot\psi_{2}, \Sigma^{'} \Rightarrow \varphi$, and the conclusion is $\lnot(\psi_{1} \vee \psi_{2}), \Sigma^{'} \Rightarrow \varphi$.
By induction hypothesis, $ \vdash_n \psi, \lnot\psi_{1}, \lnot\psi_{2}, \Sigma^{'} \Rightarrow \varphi$. By $(\lnot\vee\Imp)$, $ \vdash_{n+1}\psi, \lnot(\psi_{1} \vee \psi_{2}), \Sigma^{'} \Rightarrow \varphi$.
\end{proof}

\begin{lemma}
For any DM-structure $\Gamma$ and term $\varphi$, $\msf{G3DM}\vdash\varphi, \Gamma \Rightarrow \varphi $.
\end{lemma}
\begin{proof}
The proof is by induction on the complexity $n$ of $\varphi$.
The cases of propositional variables, constants, conjunction and disjunction are easy. Assume $\varphi= \lnot\psi$.
By induction hypothesis, we have $\vdash\psi,\Gamma\Imp\psi$. Now we prove $\vdash\lnot\psi,\Gamma\Imp\lnot\psi$ by subinduction on the complexity $m$ of $\psi$. 
When $\psi$ is a propositional variable or constant, obviously $\vdash\lnot\psi,\Gamma\Imp\lnot\psi$. We have the following remaining cases:

Case 1. $\psi = \lnot\psi'$. By induction hypothesis on $n$, $\vdash\psi',\Gamma\Imp\psi'$. By $(\lnot\lnot\Imp)$ and $(\Imp\lnot\lnot)$, $\vdash\lnot\lnot\psi',\Gamma\Imp\lnot\lnot\psi'$.

Case 2. $\psi = \psi_{1} \wedge  \psi_{2}$. By induction hypothesis on $m$, $ \vdash\lnot\psi_{1}, \Gamma \Rightarrow \lnot\psi_{1} $ and $ \vdash\lnot\psi_{2}, \Gamma \Rightarrow \lnot\psi_{2}$. By $(\Imp\lnot\wedge)$, $\vdash\lnot\psi_{1}, \Gamma \Rightarrow \lnot(\psi_{1} \wedge \psi_{2}) $ and $\vdash\lnot\psi_{2}, \Gamma \Rightarrow \lnot(\psi_{1} \wedge \psi_{2})$. By $(\lnot\wedge\Imp)$, $\vdash\lnot(\psi_{1} \wedge \psi_{2}), \Gamma \Rightarrow \lnot(\psi_{1} \wedge \psi_{2})$.

Case 3. $\psi = \psi_{1} \vee \psi_{2}$. By induction hypothesis on $m$, $ \vdash\lnot\psi_{1}, \Gamma \Rightarrow \lnot\psi_{1} $ and $ \vdash \lnot\psi_{2}, \Gamma \Rightarrow \lnot\psi_{2}$. By $(Wk)$, $\vdash\lnot\psi_{1}, \lnot\psi_{2}, \Gamma \Rightarrow \lnot\psi_{1} $ and $\vdash\lnot\psi_{1}, \lnot\psi_{2}, \Gamma \Rightarrow \lnot\psi_{2}$. By $(\lnot\vee\Imp)$, 
$\vdash\lnot(\psi_{1} \vee \psi_{2}), \Gamma \Rightarrow \lnot\psi_{1}$ and $\vdash\lnot(\psi_{1} \vee \psi_{2}), \Gamma \Rightarrow \lnot\psi_{2}$. By $(\Imp\lnot\vee)$,
$\vdash\lnot(\psi_{1} \vee \psi_{2}), \Gamma \Rightarrow \lnot(\psi_{1} \vee \psi_{2})$.
\end{proof}

\begin{lemma}[Inversion]\label{lemma:inv_dm}
For any number $n\geq 0$, the following hold in $\msf{G3DM}$:
\begin{enumerate}
\item if $ \vdash_{n} \varphi \wedge \psi, \Sigma \Rightarrow \chi$, then $ \vdash_{n} \varphi, \psi, \Sigma \Rightarrow \chi$.
\item if $ \vdash_{n} \varphi \vee \psi, \Sigma \Rightarrow \chi$, then $ \vdash_{n} \varphi, \Sigma \Rightarrow \chi$ and $ \vdash_{n} \psi, \Sigma \Rightarrow \chi$.
\item if $ \vdash_{n} \lnot(\varphi \wedge \psi), \Sigma \Rightarrow \chi$, then $ \vdash_{n} \lnot\varphi, \Sigma \Rightarrow \chi$ and $ \vdash_{n} \lnot\psi, \Sigma \Rightarrow \chi$.
\item if $ \vdash_{n} \lnot(\varphi \vee \psi), \Sigma \Rightarrow \chi$, then $ \vdash_{n} \lnot\varphi, \lnot\psi, \Sigma \Rightarrow \chi$.
\item if $ \vdash_{n} \lnot\lnot\varphi, \Sigma \Rightarrow \chi$, then $ \vdash_{n} \varphi, \Sigma \Rightarrow \chi$.
\end{enumerate}
\end{lemma}
\begin{proof}
By induction on $n$. Consider (1). The case for $n=0$ is obvious. For the inductive step, assume that $\vdash_{n + 1} \varphi \wedge \psi, \Sigma \Rightarrow \chi$ and the last rule is $(\mathtt{R})$. If $\varphi \wedge \psi$ is principal in $(\mathtt{R})$, $\vdash_{n}\varphi, \psi, \Sigma \Rightarrow \chi$.
If $\varphi \wedge \psi$ is not principal in $(\mathtt{R})$, the conclusion is obtained by applying induction hypothesis to the premiss(es) and then using $(\mathtt{R})$. The remaining cases are shown similarly.
\end{proof}

\begin{theorem}
The contraction rule
\[
\AxiomC{$\psi, \psi, \Sigma \Rightarrow \varphi$}
\RightLabel{\small $(Ctr)$}
\UnaryInfC{$\psi, \Sigma \Rightarrow \varphi$}
\DisplayProof
\]
is height-preserving admissible in $\msf{G3DM}$.
\end{theorem}
\begin{proof}
By induction on the height $n$ of derivation of $ \psi, \psi, \Sigma \Rightarrow \varphi$ in $\msf{G3DM}$. The case for $n=0$ is easy.
Assume that $\msf{G3DM}\vdash_{n + 1} \psi, \psi, \Sigma \Rightarrow \varphi$ and the last rule is $(\mathtt{R})$.
If $\psi$ is not principal in $(\mathtt{R})$, then apply induction hypothesis to the premiss(es) of $(\mathtt{R})$ and then apply $(\mathtt{R})$.
Assume that one occurrence of $\psi$ is principal in $(\mathtt{R})$. We use Lemma \ref{lemma:inv_dm} to the premiss(es) first and apply induction hypothesis. Finally, by $(\mathtt{R})$, we obtain the result. 
\end{proof}

\begin{theorem}\label{thm:cut_dm}
The cut rule
\[
\AxiomC{$ \Sigma \Rightarrow \varphi $}
\AxiomC{$ \varphi, \Theta \Rightarrow \psi $}
\RightLabel{\small $(Cut)$}
\BinaryInfC{$ \Sigma, \Theta \Rightarrow \psi $}
\DisplayProof
\]
is admissible in $\msf{G3DM}$.
\end{theorem}
\begin{proof}
By simultaneous induction on the height of derivation of the left premiss, the height of derivation of the right premiss, and the number of connectives in the cut term.
The proof is similar to Theorem \ref{thm:cut_sdm}.
\end{proof}


\subsection{Completeness}
For any DM-structure $\Sigma=(\varphi_1,\ldots,\varphi_n)$, let $\lnot\Sigma = (\lnot\varphi_1,\ldots$, $\lnot\varphi_n)$, and let $\bigwedge\Sigma = \varphi_1\wedge\ldots\wedge\varphi_n$ and $\bigvee\Sigma = \varphi_1\vee\ldots\vee\varphi_n$. We understand $\bigwedge\emptyset = \top$ and $\bigvee\emptyset=\bot$.
A DM-sequent $\Sigma \Imp \varphi$ is {\em valid} in a De Morgan algebra $\mf{A}$, notation $\mf{A}\models\Sigma \Imp \varphi$, if $\mf{A}\models\bigwedge\Sigma \Imp \varphi$. The notation $\msf{DM} \models \Sigma \Rightarrow \varphi$ means that $\Sigma \Rightarrow \varphi$ is valid in all De Morgan algebras.

\begin{lemma}\label{lemma:cp_dm}
If $\msf{G3DM} \vdash \Sigma \Imp \varphi$, then $\msf{G3SDM} \vdash \lnot\varphi\Imp \bigvee\lnot\Sigma$.
\end{lemma}

\begin{proof}
By induction on the height $n$ of derivation of $\Sigma \Imp \varphi$ in $\msf{G3DM}$. Assume that $\vdash_n \Sigma \Imp \varphi$ and the last step is $(\mathtt{R})$. 

Case 1. $(\mathtt{R})$ is an axiom. We have the following cases:

(1.1) $\varphi=p$ and $\Sigma = p,\Sigma'$ for some variable $p$. We get $\lnot p\Imp \lnot p\vee\bigvee\lnot\Sigma'$ from $\lnot p\Imp\lnot p$ by $(\Imp\vee)$. 

(1.2) $\varphi=\lnot p$ and $\Sigma = \lnot p,\Sigma'$ for some variable $p$. We get $\lnot\lnot p\Imp \lnot\lnot p\vee\bigvee\lnot\Sigma'$ from $\lnot\lnot p\Imp\lnot \lnot p$  by $(\vee\Imp)$.

(1.3) $(\mathtt{R})$ is $(\bot\Imp)$. Let $\Sigma=\bot,\Sigma'$. We get 
$\lnot\varphi\Imp\lnot\bot\vee\bigvee\lnot\Sigma'$ from $\lnot\varphi\Imp\lnot\bot$ by $(\Imp\vee)$.

(1.4) $(\mathtt{R})$ is $(\Imp\lnot\bot)$. Let $\varphi=\lnot\bot$. We get $\lnot\lnot\bot\Imp\bigvee\lnot\Sigma'$ from $\bot\Imp\bigvee\lnot\Sigma'$ by $(\lnot\lnot\Imp)$.

Case 2. $(\mathtt{R})$ is $(\wedge\Imp)$. Let $\Sigma = \psi_1\wedge\psi_2,\Sigma'$.
The premiss of $(\mathtt{R})$ is $\psi_1,\psi_2,\Sigma'\Imp\varphi$, and the conclusion is $\psi_1\wedge\psi_2,\Sigma'\Imp\varphi$.
By induction hypothesis, $\vdash\lnot\varphi\Imp \lnot\psi_1\vee\lnot\psi_2\vee \bigvee\lnot\Sigma'$. It is easy to show $\vdash\lnot\psi_1\vee\lnot\psi_2\vee\bigvee\lnot\Sigma'\Imp \lnot(\psi_1\wedge\psi_2)\vee\bigvee\lnot\Sigma'$. By $(Cut)$, $\vdash\lnot\varphi\Imp \lnot(\psi_1\wedge\psi_2)\vee\bigvee\lnot\Sigma'$.

Case 3. $(\mathtt{R})$ is $(\wedge\Imp)$. Let $\varphi = \varphi_1\wedge\varphi_2$.
The premisses of $(\mathtt{R})$ are $\Sigma\Imp\varphi_1$ and $\Sigma\Imp\varphi_2$, and the conclusion is $\Sigma\Imp\varphi_1\wedge\varphi_2$.
By induction hypothesis, $\vdash\lnot\varphi_1\Imp\bigvee\lnot\Sigma$ and $\vdash \lnot\varphi_2\Imp\bigvee\lnot\Sigma$. By $(\lnot\wedge\Imp)$, $\vdash\lnot(\varphi_1\wedge\varphi_2)\Imp\bigvee\lnot\Sigma$.

Case 4. $(\mathtt{R})$ is $(\vee\Imp)$. Let $\Sigma = \psi_1\vee\psi_2,\Sigma'$. The premisses of $(\mathtt{R})$ are $\psi_1,\Sigma\Imp\varphi$ and $\psi_2,\Sigma\Imp\varphi$, and the conclusion is 
$\psi_1\vee\psi_2,\Sigma\Imp\varphi $.
By induction hypothesis, $\vdash\lnot\varphi \Imp \lnot\psi_1\vee\bigvee\lnot\Sigma'$ and $\vdash\lnot\varphi \Imp\lnot\psi_2\vee \bigvee\lnot\Sigma'$. By $(\Imp\wedge)$, $\vdash\lnot\varphi \Imp (\lnot\psi_1\vee\bigvee\lnot\Sigma')\wedge (\lnot\psi_2\vee\bigvee\lnot\Sigma')$. It is easy to show 
$\vdash(\lnot\psi_1\vee\bigvee\lnot\Sigma')\wedge (\lnot\psi_2\vee\bigvee\lnot\Sigma')\Imp (\lnot\psi_1\wedge\lnot\psi_2)\vee \bigvee\lnot\Sigma'$. By $(Cut)$, $\vdash\lnot\varphi \Imp (\lnot\psi_1\wedge\lnot\psi_2)\vee \bigvee\lnot\Sigma'$. It is easy to show $(\lnot\psi_1\wedge\lnot\psi_2)\vee \bigvee\lnot\Sigma'\Imp \lnot(\psi_1\vee\psi_2)\vee \bigvee\lnot\Sigma'$. By $(Cut)$, $\vdash\lnot\varphi\Imp \lnot(\psi_1\vee\psi_2)\vee \bigvee\lnot\Sigma'$.

Case 5. $(\mathtt{R})$ is $(\Imp\vee)$. Let $\varphi = \varphi_1\vee\varphi_2$. The premiss of $(\mathtt{R})$ is $\Sigma\Imp\varphi_i$ and the conclusion is $\Sigma\Imp\varphi_1\vee\varphi_2$.
By induction hypothesis, $\vdash\lnot\varphi_i \Imp \bigvee\lnot\Sigma$. By $(Wk)$, $\vdash\lnot\psi_1,\lnot\psi_2 \Imp\bigvee\lnot\Sigma $. By $(\lnot\vee\Imp)$, $\vdash\lnot(\psi_1\vee\psi_2) \Imp\bigvee\lnot\Sigma $.

Case 6. $(\mathtt{R})$ is $(\lnot\wedge\Imp)$ or $(\lnot\lnot\Imp)$. The proof is similar to Case 4. 

Case 7. $(\mathtt{R})$ is $(\vee\Imp)$ or $(\lnot\lnot\Imp)$. The proof is similar to Case 5.
\end{proof}

\begin{lemma}\label{lemma:dmc1}
If $S_\msf{DM} \vdash \varphi \Rightarrow \psi$, then $\msf{G3DM} \vdash \varphi \Rightarrow \psi$.
\end{lemma}
\begin{proof}
By induction on the derivation of $\varphi \Rightarrow \psi$ in $S_\msf{DM}$. Note that the contraposition rule (CP) is admissible in $\msf{G3DM}$ by Lemma \ref{lemma:cp_dm}.
\end{proof}

\begin{lemma}\label{lemma:dmc2}
If $\msf{G3DM} \vdash \bigwedge\Sigma \Rightarrow \varphi$, then $\msf{G3DM} \vdash \Sigma \Rightarrow \varphi$.
\end{lemma}
\begin{proof}
By Lemma \ref{lemma:inv_dm} (1).
\end{proof}

\begin{lemma}\label{lemma:dmc3}
If $\msf{S_{DM}} \vdash \bigwedge\Sigma \Rightarrow \varphi$, then $\msf{G3DM} \vdash \Sigma \Rightarrow \varphi$.
\end{lemma}
\begin{proof}
By Lemma \ref{lemma:dmc1} and Lemma \ref{lemma:dmc2}.
\end{proof}

\begin{theorem}
For any DM-sequent $\Sigma\Imp\varphi$, $\msf{G3DM} \vdash \Sigma \Rightarrow \varphi$ if and only if $\msf{DM} \models \Sigma \Rightarrow \varphi$.
\end{theorem}
\begin{proof}
The soundness part is shown easily by induction on the height of derivation. For the completeness part, assume $\msf{G3DM} \not\vdash \Sigma \Rightarrow \varphi$.
By Lemma \ref{lemma:dmc3}, $S_\msf{DM} \not\vdash \bigwedge\Sigma \Rightarrow \varphi$. By the completeness of $S_\msf{DM}$, there is a De Morgan algebra $\mf{A}$ such that $\mf{A}\not\models\bigwedge\Sigma\Imp\varphi$. Therefore $\mf{A}\not\models\Sigma\Imp\varphi$. 
\end{proof}

\subsection{Interpolation and decidability}

Given a DM-sequent $\Sigma\Imp\varphi$, we say that $(\Sigma_1;\emptyset)(\Sigma_2; \varphi)$ is a {\em partition} of $\Sigma\Imp\varphi$ if the multiset union of $\Sigma_1$ and $\Sigma_2$ is equal to $\Sigma$. Let $\Sigma\Imp\varphi$ be any DM-sequent and $\msf{G3DM} \vdash \Sigma \Rightarrow \varphi$.
A term $\psi$ is called an {\em interpolant} of the partition 
$(\Sigma_1;\emptyset)(\Sigma_2; \varphi)$ of $\Sigma\Imp\varphi$ if (i) $\msf{G3DM} \vdash  \Sigma_1 \Rightarrow \psi$; (ii) $\msf{G3DM} \vdash \psi, \Sigma_2 \Rightarrow \varphi$; and (iii) $var(\psi)\sub var(\Sigma_1)\cap var(\Sigma_2, \varphi)$.

\begin{theorem}[Interpolation]
For any DM-sequent $\Sigma\Imp\varphi$, if 
$\msf{G3DM} \vdash \Sigma \Rightarrow \psi$, then any partition of $\Sigma\Imp\varphi$ has an interpolant.
\end{theorem}

\begin{proof}
By induction on the height $n$ of derivation of $\Sigma \Rightarrow \psi$. The proof is quite similar with the proof of Theorem \ref{thm:int_sdm}.
\end{proof}

For the decidability of derivability in $\msf{G3DM}$, the proof is similar with the proof of Theorem \ref{dec:sdm}.
We need a new definition of the weight of terms.

\begin{definition}
The {\em DM-weight} $\mu(\varphi)$ of a term $\varphi$ is defined as follows:
\begin{align*}
\mu(p) &= \mu(\bot) = 1 & 
\mu(\lnot\varphi) &= \mu(\varphi) + 1\\
\mu(\varphi \vee \psi) &= \mu(\varphi) + \mu(\psi) + 2 &
\mu(\varphi \wedge \psi) &= \mu(\varphi) + \mu(\psi) + 2.
\end{align*}

For any DM-structure $\Sigma=(\varphi_1,\ldots,\varphi_n)$, the DM-weight of $\Sigma$ is defined as $\mu(\Sigma) = \mu(\varphi_1) + \ldots + \mu(\varphi_n)$.
The DM-weight of a DM-sequent $\Sigma\Imp\varphi$ is defined as $\mu(\Sigma \Rightarrow \varphi) = \mu(\Sigma) + \mu(\varphi)$.
\end{definition}

It is easy to check that, in any logical rule in $\msf{G3DM}$, the weight of each premiss is strictly less than the weight of the conclusion. Then we obtain the following decidability result.

\begin{theorem}[Decidability]
The derivability of a DM-sequent $\Sigma \Imp\varphi$ in the calculus $\msf{G3DM}$ is decidable.
\end{theorem}

\section{Some Embedding Theorems}
In this section, we shall prove some embedding theorems. We shall show that (i) $\msf{G3DM}$ is embedded into $\msf{G3SDM}$; (ii) $\msf{G3SDM}$ is embedded into $\msf{G3ip}$ for intuitionistic propositional logic; (iii) 
$\msf{G3SDM}$ is embedded into $\msf{G3ip}$+\textbf{Gem-at} for classical propositional logic.

\subsection{Embedding of $\msf{G3DM}$ into $\msf{G3SDM}$}
In this subsection, we shall show that $\msf{G3DM}$ is embedded into $\msf{G3SDM}$ by G\"odel-Gentzen translation, which is also an embedding from classical logic into intuitionistic logic (cf.~\cite{Gentzen1936,Godel1933}).
Following from this embedding result, we shall show further that Glivenko's double negation translation embeds $\msf{G3DM}$ into $\msf{G3SDM}$.

\begin{definition}\label{def: trans}
The G\"odel-Gentzen translation $f: \mc{T} \imp \mc{T}$ is a function defined as follows:
\begin{align*}
f(\bot) &= \bot & f(p) &= \lnot\lnot p \\
f(\lnot\varphi) &= \lnot f(\varphi) & f(\varphi \wedge \psi) &= f(\varphi) \wedge f(\psi)\\
f(\varphi \vee \psi) &= \lnot\lnot(f(\varphi) \vee f(\psi)).
\end{align*}
For any DM-structure $\Sigma = (\varphi_1, \ldots, \varphi_n)$, define $f(\Sigma) = f(\varphi_1)\wedge \ldots \wedge f(\varphi_n)$. 
\end{definition}

\begin{lemma}\label{lemma:double}
For any term $\varphi \in \mc{T}$, the sequents $f(\varphi) \Rightarrow \lnot\lnot f(\varphi)$ and $\lnot\lnot f(\varphi) \Rightarrow f(\varphi)$ are derivable in $\msf{G3SDM}$.
\end{lemma}
\begin{proof}
By induction on the complexity of $\varphi$. The atomic case is obvious. We have the following remaining cases:

Case 1. $\varphi:= \lnot\psi$. Then $f(\varphi) = \lnot f(\psi)$ and $\lnot\lnot f(\varphi)= \lnot\lnot\lnot f(\psi)$. By induction hypothesis,
$\vdash f(\psi) \Rightarrow \lnot\lnot f(\psi)$ and $\vdash\lnot\lnot f(\psi) \Rightarrow  f(\psi)$.
By (CP), $\vdash\lnot f(\psi) \Rightarrow \lnot\lnot\lnot f(\psi)$ and $\vdash\lnot\lnot\lnot f(\psi) \Rightarrow \lnot f(\psi)$.

Case 2. $\varphi:= \varphi_1\wedge\varphi_2$. Then $f(\varphi) = f(\varphi_1)\wedge f(\varphi_2)$ and $\lnot\lnot f(\varphi)= \lnot\lnot (f(\varphi_1)\wedge f(\varphi_2))$. We show $\vdash f(\varphi_1)\wedge f(\varphi_2) \Rightarrow \lnot\lnot (f(\varphi_1)\wedge f(\varphi_2))$ as follows. By induction hypothesis, $\vdash f(\varphi_1)\Rightarrow \lnot\lnot f(\varphi_1)$ and $\vdash f(\varphi_2)\Rightarrow \lnot\lnot f(\varphi_2)$.
By Lemma \ref{exchange} (1),
$\vdash f(\varphi_1)\Rightarrow *\lnot f(\varphi_1)$ and $\vdash f(\varphi_2)\Rightarrow *\lnot f(\varphi_2)$.
By $(Wk)$, $\vdash f(\varphi_1),f(\varphi_2)\Rightarrow *\lnot f(\varphi_1)$ and $\vdash f(\varphi_1),f(\varphi_2)\Rightarrow *\lnot f(\varphi_2)$. By $(\Rightarrow*\lnot\wedge)$, $\vdash f(\varphi_1),f(\varphi_2)\Rightarrow *\lnot (f(\varphi_1)\wedge f(\varphi_2))$. By $(\wedge\Rightarrow)$ and $(\Rightarrow\lnot)$, $\vdash f(\varphi_1)\wedge f(\varphi_2)\Rightarrow \lnot\lnot (f(\varphi_1)\wedge f(\varphi_2))$.

Now we show $\vdash \lnot\lnot (f(\varphi_1)\wedge f(\varphi_2))\Rightarrow f(\varphi_1)\wedge f(\varphi_2)$ as follows. By induction hypothesis, $\vdash\lnot\lnot f(\varphi_1)\Rightarrow f(\varphi_1)$ and $\vdash\lnot\lnot f(\varphi_2)\Rightarrow f(\varphi_2)$.
By Lemma \ref{exchange} (2), 
$\vdash*\lnot f(\varphi_1)\Rightarrow f(\varphi_1)$ and $\vdash*\lnot f(\varphi_2)\Rightarrow f(\varphi_2)$.
By $(Wk)$, $\vdash*\lnot f(\varphi_1), *\lnot f(\varphi_2)\Rightarrow f(\varphi_1)$ and 
$\vdash*\lnot f(\varphi_1), *\lnot f(\varphi_2)\Rightarrow f(\varphi_2)$.
By $(\Rightarrow\wedge)$, $\vdash*\lnot f(\varphi_1), *\lnot f(\varphi_2)\Rightarrow f(\varphi_1)\wedge f(\varphi_2)$. By $(*\lnot\wedge\Rightarrow)$, $\vdash*\lnot (f(\varphi_1)\wedge f(\varphi_2))\Rightarrow f(\varphi_1)\wedge f(\varphi_2)$. By $(\lnot\Rightarrow)$, $\vdash\lnot\lnot (f(\varphi_1)\wedge f(\varphi_2))\Rightarrow f(\varphi_1)\wedge f(\varphi_2)$.

Case 3. $\varphi:= \varphi_1\vee\varphi_2$. Then $f(\varphi) = \lnot\lnot(f(\varphi_1)\vee f(\varphi_2))$. We show $\vdash\lnot \lnot(f(\varphi_1)\vee f(\varphi_2)) \Rightarrow \lnot\lnot\lnot\lnot(f(\varphi_1)\vee f(\varphi_2))$ 
as follows. By Lemma \ref{lemma:id}, $\vdash\lnot\lnot(f(\varphi_1)\vee f(\varphi_2)) \Rightarrow \lnot\lnot(f(\varphi_1)\vee f(\varphi_2))$. By Lemma \ref{exchange} (1), $\vdash\lnot\lnot(f(\varphi_1)\vee f(\varphi_2)) \Rightarrow *\lnot(f(\varphi_1)\vee f(\varphi_2))$.
By $(\Rightarrow*\lnot\lnot)$, $\vdash\lnot\lnot(f(\varphi_1)\vee f(\varphi_2)) \Rightarrow *\lnot\lnot\lnot(f(\varphi_1)\vee f(\varphi_2))$. By $(\Rightarrow \lnot)$, $\vdash\lnot\lnot(f(\varphi_1)\vee f(\varphi_2)) \Rightarrow \lnot\lnot\lnot\lnot(f(\varphi_1)\vee f(\varphi_2))$.

Now we show $\vdash\lnot\lnot\lnot\lnot(f(\varphi_1)\vee f(\varphi_2))
\Rightarrow \lnot \lnot(f(\varphi_1)\vee f(\varphi_2))$ as follows. By Lemma \ref{lemma:id}, $\vdash\lnot\lnot(f(\varphi_1)\vee f(\varphi_2)) \Rightarrow \lnot\lnot(f(\varphi_1)\vee f(\varphi_2))$. By Lemma \ref{exchange} (2), $\vdash*\lnot(f(\varphi_1)\vee f(\varphi_2)) \Rightarrow \lnot\lnot(f(\varphi_1)\vee f(\varphi_2))$. By $(*\lnot\lnot\Rightarrow)$ and $(\lnot\Rightarrow)$, $\vdash\lnot\lnot\lnot\lnot(f(\varphi_1)\vee f(\varphi_2))
\Rightarrow \lnot \lnot(f(\varphi_1)\vee f(\varphi_2))$.
\end{proof}

\begin{corollary}\label{coro:double}
For any DM-sequent $\Sigma\Imp\varphi$, $f(\Sigma) \Rightarrow  \lnot\lnot f(\Sigma) $ and $\lnot\lnot f(\Sigma) \Rightarrow f(\Sigma)$ are derivable in $\msf{G3SDM}$.
\end{corollary}

\begin{lemma}\label{lemma:dm_trans}
For any DM-sequent $\Sigma\Imp\varphi$, 
$\msf{G3DM} \vdash  \Sigma \Imp \varphi$ if and only if $\msf{G3DM} \vdash f(\Sigma) \Imp f(\varphi)$.
\end{lemma}
\begin{proof}
The `only if' part is shown by induction on the height of derivation of $\Sigma\Imp\varphi$ in $\msf{G3DM}$. When $n=0$, the sequent $\Sigma\Imp\varphi$ is an axiom. Obviously $\msf{G3DM} \vdash f(\Sigma) \Imp f(\varphi)$. Let $n>0$. Assume that $\msf{G3DM}\vdash_n\Sigma\Imp\varphi$ is obtained by $(\mathtt{R})$. By induction hypothesis and the definition of $f$, one can easily obtain $\msf{G3DM} \vdash f(\Sigma) \Imp f(\varphi)$. 

The `if' part is shown by induction on the height $n$ of derivation of $f(\Sigma) \Imp f(\varphi)$ in $\msf{G3DM}$. Let $f(\Sigma) \Imp f(\varphi)$ be obtained by $(\mathtt{R})$. Note that $f(\Sigma) \Imp f(\varphi)$ can not be an axiom in $\msf{G3DM}$. Let $n>0$. Assume that $\msf{G3DM}\vdash_n f(\Sigma) \Imp f(\varphi)$ is obtained by $(\mathtt{R})$. By induction hypothesis and the definition of $f$, one can easily get the result. For example, assume $(\mathtt{R})=(\wedge\Imp)$. Let the premiss be $f(\psi_1), f(\psi_2), f(\Sigma') \Imp f(\varphi)$ and the conclusion be $f(\psi_1) \wedge f(\psi_2), f(\Sigma') \Imp f(\varphi)$. By induction hypothesis, $\vdash_{n -1} \psi_1, \psi_2, \Sigma' \Imp \varphi$. By $(\wedge\Imp)$, $\vdash_{n} \psi_1 \wedge \psi_2, \Sigma' \Imp \varphi$. By the definition of $f$, $f(\psi_1 \wedge \psi_2) = f(\psi_1) \wedge f(\psi_2)$, as required.
\end{proof}

For any SDM-sequent $\Gamma = \alpha_1,\ldots,\alpha_n$, define $\Gamma^t = t(\alpha_1),\ldots, t(\alpha_n)$. Then we have the following lemma.

\begin{lemma}\label{lemma:dm_trans2}
For any SDM-sequent $\Gamma\Imp\alpha$, if $\msf{G3SDM} \vdash\Gamma\Imp\alpha$, then $\msf{G3DM} \vdash\Gamma^t \Imp t(\alpha)$.
\end{lemma}
\begin{proof}
By Lemma \ref{exchange}.
\end{proof}

\begin{lemma}\label{prop:seq}
The following sequents are derivable in $\msf{G3SDM}$:
$(1)$ $ \lnot(\varphi \wedge \psi) \Rightarrow \lnot\lnot(\lnot\varphi \vee \lnot\psi)$;
$(2)$ $ \lnot\lnot(\lnot\varphi \vee \lnot\psi) \Rightarrow \lnot(\varphi \wedge \psi)$;
$(3)$ $ \lnot(\lnot\varphi \wedge \lnot\psi) \Rightarrow \lnot\lnot(\varphi \vee \psi)$;
$(4)$ $\lnot\lnot(\varphi \vee \psi) \Rightarrow  \lnot(\lnot\varphi \wedge \lnot\psi)$.
\end{lemma}
\begin{proof}
We show only (1) and (2). (3) and (4) are shown easily. For (1), it suffices to derive (i) $\vdash\lnot(\varphi \wedge \psi) \Rightarrow \lnot\lnot\lnot(\varphi \wedge \psi)$ and (ii) $\vdash\lnot\lnot\lnot(\varphi \wedge \psi) \Rightarrow \lnot\lnot(\lnot\varphi \vee \lnot\psi)$, and then apply $(Cut)$. (i) is shown easily.
For (ii), first, it is easy to get $\vdash*\lnot\varphi, *\lnot\psi \Rightarrow \lnot\lnot\varphi$ and $\vdash*\lnot\psi, *\lnot\varphi \Rightarrow \lnot\lnot\psi$. By $(\Imp\wedge)$, $\vdash*\lnot\psi, *\lnot\varphi \Rightarrow \lnot\lnot\varphi \wedge \lnot\lnot\psi $.
Second, it is easy to show $\vdash\lnot\lnot\varphi, \lnot\lnot\psi \Rightarrow *\lnot\varphi$ and $\vdash\lnot\lnot\psi, \lnot\lnot\varphi \Rightarrow *\lnot\psi$. By $(\Imp*\lnot\wedge)$, $\vdash\lnot\lnot\psi, \lnot\lnot\varphi \Rightarrow *\lnot(\varphi \wedge \psi)$. By $(\wedge\Imp)$ and $(\Imp\lnot)$, $\vdash\lnot\lnot\varphi \wedge \lnot\lnot\psi \Rightarrow \lnot\lnot(\varphi \wedge \psi)$.
Then we have $\vdash*\lnot\psi, *\lnot\varphi \Rightarrow \lnot\lnot(\varphi \wedge \psi)$ by $(Cut)$.
By $(*\vee\Imp)$, $\vdash*(\lnot\varphi \vee \lnot\psi) \Rightarrow \lnot\lnot(\varphi \wedge \psi)$.
By $(\lnot\Imp)$, $\vdash\lnot(\lnot\varphi \vee \lnot\psi) \Rightarrow \lnot\lnot(\varphi \wedge \psi)$.
By (CP), $\vdash\lnot\lnot\lnot(\varphi \wedge \psi) \Rightarrow \lnot\lnot(\lnot\varphi \vee \lnot\psi)$.

(2) It suffices to show (i) $\vdash\lnot\lnot(\lnot\varphi \vee \lnot\psi)\Rightarrow \lnot(\lnot\lnot\varphi \wedge \lnot\lnot\psi)$ and (ii) $\vdash\lnot(\lnot\lnot\varphi \wedge \lnot\lnot\psi) \Rightarrow \lnot(\varphi \wedge \psi)$, and then apply $(Cut)$.
For (i), we have the following derivation: it is easy to show $\vdash\lnot\lnot\varphi \Rightarrow *\lnot\varphi$ and 
$\vdash\lnot\lnot\varphi, \lnot\lnot\psi \Rightarrow *\lnot\psi$. By $(\Imp*\vee)$, $\vdash\lnot\lnot\varphi, \lnot\lnot\psi \Rightarrow *(\lnot\varphi \vee \lnot\psi)$. By $(\Imp\lnot)$ and $(\wedge\Imp)$, $\vdash\lnot\lnot\varphi \wedge \lnot\lnot\psi \Rightarrow \lnot(\lnot\varphi \vee \lnot\psi)$. By (CP), $\vdash\lnot\lnot(\lnot\varphi \vee \lnot\psi)\Rightarrow \lnot(\lnot\lnot\varphi \wedge \lnot\lnot\psi)$.

For (ii), first, we have $\vdash*\lnot\varphi, *\lnot\psi \Rightarrow \lnot\lnot\varphi$ and $\vdash*\lnot\varphi, *\lnot\psi \Rightarrow \lnot\lnot\psi$. By $(\Imp\wedge)$, $\vdash*\lnot\varphi, *\lnot\psi \Rightarrow \lnot\lnot\varphi \wedge \lnot\lnot\psi$. By $(*\lnot\wedge\Imp)$, $\vdash*\lnot(\varphi \wedge \psi) \Rightarrow \lnot\lnot\varphi \wedge \lnot\lnot\psi$. By $(\lnot\Imp)$, $\vdash\lnot\lnot(\varphi \wedge \psi) \Rightarrow \lnot\lnot\varphi \wedge \lnot\lnot\psi$. By (CP), $\vdash\lnot(\lnot\lnot\varphi \wedge \lnot\lnot\psi) \Rightarrow \lnot\lnot\lnot(\varphi \wedge \psi)$.
Second, it is easy to show $\vdash\lnot\lnot\lnot(\varphi \wedge \psi) \Rightarrow \lnot(\varphi \wedge \psi)$.
By $(Cut)$, $\vdash\lnot(\lnot\lnot\varphi \wedge \lnot\lnot\psi) \Rightarrow \lnot(\varphi \wedge \psi)$.
\end{proof}

\begin{theorem}\label{thm:dm_sdm}
For any DM-sequent $\Sigma\Imp\varphi$,
$\msf{G3DM} \vdash \Sigma \Imp \varphi$ if and only if $\msf{G3SDM} \vdash f(\Sigma) \Imp f(\varphi)$.
\end{theorem}
\begin{proof}
For the `if' part, assume $\msf{G3SDM} \vdash f(\Sigma) \Imp f(\varphi)$. By Lemma \ref{lemma:dm_trans2}, $\msf{G3DM} \vdash f(\Sigma)^t \Imp t(f(\varphi))$. Since $f(\varphi)$ is a term and $f(\Sigma)$ is a DM-structure, $f(\Sigma)^t = f(\Sigma)$ and $t(f(\varphi))= f(\varphi)$. Then $\msf{G3DM} \vdash f(\Sigma) \Imp f(\varphi)$. By Lemma \ref{lemma:dm_trans}, $\msf{G3DM} \vdash \Sigma \Imp \varphi$.

The `only if' part is shown by induction on the height $n$ of derivation of $\Sigma \Imp \varphi$ in $\msf{G3DM}$. The case $n = 0$ is quite easy. Assume that $n>0$ and $\Sigma \Imp \varphi$ is obtained by a rule $(\mathtt{R})$. Then we have the following cases:

Case 1. $(\mathtt{R})$ is $(\wedge \Imp)$ or $(\Imp\wedge)$. By the definition of $f$, using induction hypothesis, we can easily obtain the required result.

Case 2. $(\mathtt{R})$ is $(\vee \Imp)$. Let $\Sigma = \psi_1 \vee \psi_2, \Sigma'$. The premisses of $(\mathtt{R})$ are $\psi_1, \Sigma' \Imp \varphi$ and $\psi_2, \Sigma' \Imp \varphi$, and the conclusion is $\psi_1 \vee \psi_2, \Sigma'\Imp\varphi$.
By induction hypothesis, 
(i) $\vdash f(\psi_1)\wedge f(\Sigma') \Imp t(\varphi)$ and (ii) $\vdash f(\psi_1)\wedge f(\Sigma') \Imp f(\varphi)$. 
By (CP), (iii) $\vdash \lnot f(\varphi) \Imp \lnot(f(\psi_1)\wedge f(\Sigma'))$ and (iv) $\vdash\lnot f(\varphi) \Imp \lnot(f(\psi_2)\wedge f(\Sigma'))$. 
By $(\Imp\wedge)$, $\vdash\lnot f(\varphi)\Imp \lnot(f(\psi_1)\wedge f(\Sigma')) \wedge \lnot(f(\psi_2)\wedge f(\Sigma'))$. By (CP), 
$\vdash\lnot(\lnot(f(\psi_1)\wedge f(\Sigma')) \wedge \lnot(f(\psi_2)\wedge f(\Sigma')))\Imp\lnot\lnot f(\varphi)$.
By Lemma \ref{prop:seq} (4) and $(Cut)$, 
$\vdash\lnot\lnot((f(\psi_1)\wedge f(\Sigma')) \vee (f(\psi_2)\wedge f(\Sigma'))) \Imp \lnot\lnot f(\varphi)$.
Clearly $\vdash f(\Sigma')\wedge (f(\psi_1)\vee f(\psi_2))\Imp (f(\psi_1)\wedge f(\Sigma')) \vee (f(\psi_2)\wedge f(\Sigma'))$.
By twice applications of (CP), 
$\vdash\lnot\lnot(f(\Sigma')\wedge (f(\psi_1)\vee f(\psi_2)))\Imp \lnot\lnot((f(\psi_1)\wedge f(\Sigma')) \vee (f(\psi_2)\wedge f(\Sigma')))$.
By $(Cut)$, $\vdash\lnot\lnot(f(\Sigma')\wedge (f(\psi_1)\vee f(\psi_2))) \Imp \lnot\lnot f(\varphi)$.
Obviously $\vdash\lnot\lnot f(\Sigma')\wedge \lnot\lnot (f(\psi_1)\vee f(\psi_2))\Imp \lnot\lnot(f(\Sigma')\wedge (f(\psi_1)\vee f(\psi_2)))$.
By $(Cut)$, 
$\vdash\lnot\lnot f(\Sigma')\wedge \lnot\lnot (f(\psi_1)\vee f(\psi_2))\Imp 
\lnot\lnot f(\varphi)$.
By the definition of $f$, 
$\vdash\lnot\lnot f(\Sigma')\wedge f(\psi_1\vee \psi_2)\Imp 
\lnot\lnot f(\varphi)$.
By Lemma \ref{lemma:double} and Corollary \ref{coro:double}, using $(Cut)$, we have $\vdash f(\Sigma')\wedge f(\psi_1\vee \psi_2)\Imp f(\varphi)$.
By the definition of $f$, $\vdash f(\psi_1\vee \psi_2,\Sigma')\Imp f(\varphi)$.
Case 3. $(\mathtt{R})$ is $(\Imp\vee)$. Let $\varphi = \varphi_1 \vee \varphi_2$. Without loss of generality, assume that the premiss of $(\mathtt{R})$ is $\Sigma \Imp \varphi_1$. By induction hypothesis, $\vdash f(\Sigma) \Imp f(\varphi_1)$. By (CP), $\vdash\lnot f(\varphi_1) \Imp \lnot f(\Sigma)$. By $(Wk)$, $\vdash\lnot f(\varphi_1), \lnot f(\varphi_2) \Imp \lnot f(\Sigma)$. By $(\wedge\Imp)$, $\vdash \lnot f(\varphi_1) \wedge \lnot f(\varphi_2) \Imp \lnot f(\Sigma)$. By (CP), $\vdash\lnot\lnot f(\Sigma) \Imp \lnot(\lnot f(\varphi_1) \wedge \lnot f(\varphi_2)) $. By Corollary \ref{coro:double}, $\vdash f(\Sigma) \Imp \lnot\lnot f(\Sigma)$.
By $(Cut)$, 
$\vdash f(\Sigma) \Imp \lnot(\lnot f(\varphi_1) \wedge \lnot f(\varphi_2))$.
By Lemma \ref{prop:seq} (3), 
$\vdash\lnot(\lnot f(\varphi_1) \wedge \lnot f(\varphi_2)) \Imp \lnot\lnot(f(\varphi_1) \vee f(\varphi_2))$.
By $(Cut)$, $\vdash f(\Sigma) \Imp \lnot\lnot(f(\varphi_1) \vee f(\varphi_2))$.

Case 4. $(\mathtt{R})$ is $(\lnot\vee\Imp)$ or $(\Imp\lnot\vee)$. The two cases are similar. Here we sketch the proof for the case $(\mathtt{R})=(\lnot\vee\Imp)$. Let $\Sigma = \lnot(\psi_1 \vee \psi_2), \Sigma'$ and the premiss of $(\mathtt{R})$ be $\lnot \psi_1, \lnot\psi_2, \Sigma'\Imp \varphi$. By induction hypothesis, 
$\vdash\lnot f(\psi_1) \wedge \lnot f(\psi_2) \wedge f(\Sigma') \Imp f(\varphi)$. Clearly $\vdash\lnot(f(\psi_1) \vee f(\psi_2)) \wedge f(\Sigma') \Imp \lnot f(\psi_1) \wedge \lnot f(\psi_2) \wedge f(\Sigma')$.
By $(Cut)$, $\vdash\lnot(f(\psi_1) \vee f(\psi_2)) \wedge f(\Sigma') \Imp f(\varphi)$.
We have 
$\vdash\lnot\lnot\lnot(f(\psi_1) \vee f(\psi_2)) \wedge f(\Sigma')\Imp\lnot(f(\psi_1) \vee f(\psi_2)) \wedge f(\Sigma')$. By $(Cut)$, $\vdash\lnot\lnot\lnot(f(\psi_1) \vee f(\psi_2)) \wedge f(\Sigma') \Imp f(\varphi)$.

Case 5. $(\mathtt{R})$ is $(\lnot\wedge\Imp)$ or $(\Imp\lnot\wedge)$. The proof is similar to Case 2 and Case 3. Note that Lemma \ref{prop:seq} (1) is applied in the proof.

Case 6. $(\mathtt{R})$ is $(\lnot\lnot\Imp)$ or $(\Imp\lnot\lnot)$. The two cases are similar. Consider the case that $(\mathtt{R})$ is $(\lnot\lnot\Imp)$. Let $\Sigma = \lnot\lnot\psi, \Sigma'$ and the premiss of $(\mathtt{R})$ be $\psi, \Sigma'\Imp \varphi$.
By induction hypothesis, $\vdash f(\psi)\wedge f(\Sigma')\Imp f(\varphi)$. 
By Lemma \ref{lemma:double} and $(Cut)$, $\vdash\lnot\lnot f(\psi)\wedge f(\Sigma')\Imp f(\psi)\wedge f(\Sigma')$.
\end{proof}

For any DM-structure $\Sigma= (\varphi_1,\ldots,\varphi_n)$, let $\lnot\lnot\Sigma$ be the DM-structure $\lnot\lnot\varphi_1$, $\ldots,\lnot\lnot\varphi_n$.

\begin{lemma}\label{lemma:sdm_double_negation}
For any term $\varphi\in \mc{T}$, $f(\varphi) \Imp \lnot\lnot\varphi$ and $\lnot\lnot\varphi\Imp f(\varphi)$ are derivable in $\msf{G3SDM}$.
\end{lemma}
\begin{proof}
By induction on the complexity of $\varphi$. The atomic case is easy. We have the following remaining cases:

Case 1. $\varphi = \varphi_1\wedge\varphi_2$. We have $f(\varphi) = f(\varphi_1)\wedge f(\varphi_2)$. By induction hypothesis, we have $\vdash f(\varphi_1) \Imp \lnot\lnot\varphi_1$ and $\vdash f(\varphi_2)\Imp \lnot\lnot\varphi_2$. By $(Wk)$, $(\wedge\Imp)$ and $(\Imp\wedge)$, $\vdash f(\varphi_1)\wedge f(\varphi_2) \Imp \lnot\lnot\varphi_1\wedge\lnot\lnot\varphi_2$. Since $\vdash \lnot\lnot\varphi_1\wedge\lnot\lnot\varphi_2\Imp \lnot\lnot(\varphi_1\wedge\varphi_2)$, we have $\vdash f(\varphi_1)\wedge f(\varphi_2)\Imp \lnot\lnot(\varphi_1\wedge\varphi_2)$ by $(Cut)$.
The proof of $\vdash\lnot\lnot(\varphi_1\wedge\varphi_2)\Imp f(\varphi_1)\wedge f(\varphi_2)$ is similar.

Case 2. $\varphi = \varphi_1\vee\varphi_2$. We have $f(\varphi) = \lnot\lnot (f(\varphi_1)\vee f(\varphi_2))$.
We show $\vdash \lnot\lnot (f(\varphi_1)\vee f(\varphi_2))\Imp \lnot\lnot(\varphi_1\vee\varphi_2)$.
By induction hypothesis, $\vdash f(\varphi_1)\Imp \lnot\lnot\varphi_1$ and $\vdash f(\varphi_2)\Imp \lnot\lnot\varphi_2$.
Then it is easy to show $\vdash \lnot\lnot (f(\varphi_1)\vee f(\varphi_2))\Imp \lnot\lnot (\lnot\lnot\varphi_1\vee \lnot\lnot\varphi_2)$. It suffices to show $\vdash \lnot\lnot (\lnot\lnot\varphi_1\vee \lnot\lnot\varphi_2)\Imp \lnot\lnot(\varphi_1\vee\varphi_2)$. This sequent follows from Lemma \ref{prop:seq}. The proof of $\vdash \lnot\lnot(\varphi_1\vee\varphi_2)\Imp \lnot\lnot (f(\varphi_1)\vee f(\varphi_2))$ is similar.
\end{proof}

\begin{lemma}\label{lemma:double}
For any DM-structure $\Sigma\Imp\varphi$, $\msf{G3SDM}\vdash \lnot\lnot\Sigma\Imp\lnot\lnot\varphi$ if and only if $\msf{G3SDM}\vdash f(\Sigma)\Imp f(\varphi)$.
\end{lemma}
\begin{proof}
By Lemma \ref{lemma:sdm_double_negation} and $(Cut)$.
\end{proof}

\begin{theorem}\label{thm:Gli}
For any DM-sequent $\Sigma\Imp\varphi$, $\msf{G3DM}\vdash\Sigma\Imp\varphi$ if and only if $\msf{G3SDM}\vdash\lnot\lnot\Sigma\Imp\lnot\varphi$.
\end{theorem}
\begin{proof}
By Theorem \ref{thm:dm_sdm} and Lemma \ref{lemma:double}.
\end{proof}

\subsection{Embedding of $\msf{G3SDM}$ into $\msf{G3ip}$}
In this subsection, we shall show that $\msf{G3SDM}$ can be embedded into $\msf{G3ip}$. To define the language of intuitionistic logic, a new set of propositional variables is needed. Define a binary relation $\equiv$ on $\mc{T}$ by:
\begin{center}
$\varphi\equiv\psi$ if and only if $\msf{G3SDM}\vdash \varphi\Rightarrow \psi$ and $\msf{G3SDM}\vdash \psi\Rightarrow \varphi$.
\end{center}
Clearly $\equiv$ is an equivalence relation on $\mc{T}$. Let $|\varphi|_\equiv$ be the equivalence class of $\varphi$ and $\mc{T}/_\equiv = \{|\varphi|_\equiv\mid \varphi\in\mc{T}\}$.
Let $\la\sigma_\xi\ra_{\xi<\kappa}$ be an enumeration of all equivalence classes in $\mc{T}/_\equiv$. Let $\Xi_1 = \{p_\xi\mid \xi<\kappa\}$ be a set of new propositional variables. Let $\Xi_2 = \{p'\mid p\in \Xi\}$ and $\Xi_3 = \{p''\mid p\in \Xi\}$. Define $\Xi^I = \Xi\cup \Xi_1\cup \Xi_2\cup \Xi_3$.

The language of intuitionistic logic consists of the set of propositional variables $\Xi^I$, constants $\bot$, and propositional connectives $\wedge, \vee$ and $\supset$.
The set of all Int-terms $\mc{T}^{I}$ is defined inductively as follows:
\[
\mc{T}^I \ni \alpha:= p\mid\bot \mid (\theta\wedge\theta)\mid (\theta\vee\theta)\mid (\theta\supset\theta)
\]
where $p\in\Xi^I$. Define $\top:= \bot\supset \bot$ and $\neg\theta:=\theta\supset\bot$. An Int-structure is a finite multiset of Int-terms. Int-structures are denoted by $X$, $Y$, $Z$ etc. An Int-sequent is $X\Imp \theta$ where $X$ is an Int-structure and $\theta$ is an Int-term.

\begin{definition}[cf.~\cite{negr:stru01}]
The sequent calculus $\msf{G3ip}$ for intuitionistic logic consists of the following axioms and rules:
\begin{enumerate}
\item Axioms:
\quad\quad$(Id)~p, X\Imp p
\quad
(\bot\mrm{L})~\bot,X\Imp \theta$
\item Logical rules:
\[
\frac{\theta,\delta,X\Imp \gamma}{\theta\wedge\delta, X\Imp \gamma}{(\wedge\mrm{L})}
\quad
\frac{X\Imp\theta\quad X\Imp\delta}{X\Imp\theta\wedge\delta}{(\wedge\mrm{R})}
\]
\[
\frac{\theta,X\Imp\gamma\quad\delta, X\Imp\gamma}{\theta\vee\delta, X\Imp\gamma}{(\vee\mrm{L})}
\quad
\frac{X\Imp\theta_i}{X\Imp \theta_1\vee\theta_2}{(\vee\mrm{R})(i=1,2)}
\]
\[
\frac{\theta\supset\delta,X\Imp \theta\quad\delta, X\Imp\gamma}{\theta\supset\delta,X\Imp\gamma}{(\supset\mrm{L})}
\quad
\frac{\theta,X\Imp\delta}{X\Imp\theta\supset\delta}{(\supset\mrm{R})}
\]
\end{enumerate}
\end{definition}

\begin{definition}
The translation $k: \mc{T} \imp \mc{T}^I$ is defined as below:
\begin{align*}
k(p) &= p, \quad k(\bot) = \bot,\\
k(\varphi \wedge \psi) &= k(\varphi) \wedge k(\psi),\quad k(\varphi \vee \psi) = k(\varphi) \vee k(\psi), \\
k(\lnot p) &= p',\quad k(\lnot\bot) = \top,
\end{align*}
\begin{align*}
k(\lnot(\varphi \vee \psi)) &= k(\lnot\varphi) \wedge k(\lnot\psi), \\
k(\lnot(\varphi \wedge \psi)) &= 
\begin{cases}
k(\lnot\lnot(\varphi' \vee \psi')),~\mrm{if}~ \varphi = \lnot\varphi'~\mrm{and}~\psi = \lnot\psi'.\\
p_\xi,~\mrm{if}~\lnot(\varphi \wedge \psi)\in\sigma_\xi\in \mc{T}/_\equiv.
\end{cases}
\\
k(\lnot\lnot p) &= p'', \quad k(\lnot\lnot \bot) = \bot, \\
k(\lnot\lnot(\varphi \wedge \psi)) &= k(\lnot\lnot\varphi)\wedge k(\lnot\lnot\psi),\quad
k(\lnot\lnot\lnot\varphi) = k(\lnot\varphi),
\\
k(\lnot\lnot(\varphi \vee \psi)) &= 
\begin{cases}
k(\lnot(\varphi' \wedge \psi')),\mrm{if}~\varphi = \lnot\varphi'~\mrm{and}~\psi = \lnot\psi'.\\
p_\zeta,~\mrm{if}~\lnot\lnot(\varphi \vee \psi)\in \sigma_\zeta\in \mc{T}/_\equiv.
\end{cases}
\end{align*}

For any basic SDM-structure $*\varphi$, define $k(*\varphi) = k(\lnot\varphi)$. For any SDM-structure $\Gamma = \alpha_1, \ldots, \alpha_n$, define $k(\Gamma) = k(\alpha_1), \ldots, k(\alpha_n)$. 
\end{definition}

\begin{theorem}\label{thm:emb_into_int}
For any SDM-sequent $\Gamma\Imp\varphi$, $\msf{G3SDM} \vdash \Gamma \Imp \varphi$ if and only if $\msf{G3ip} \vdash k(\Gamma) \Imp k(\varphi)$.
\end{theorem}
\begin{proof}
The `only if' direction is shown by induction on the height $n$ of derivation of $\Gamma\Imp\varphi$ in $\msf{G3SDM}$. When $n=0$, $\Gamma\Imp\varphi$ is an axiom. Obviously $\msf{G3ip} \vdash k(\Gamma) \Imp k(\varphi)$. Let $n>0$. Assume that $\msf{G3SDM}\vdash_n\Gamma\Imp\varphi$ is obtained by $(\mathtt{R})$.
Note that $(\mathtt{R})$ cannot be $(*)$ because the succedent of $\Gamma\Imp\varphi$ is a term. By induction hypothesis and the definition of $k$, we get $\msf{G3ip} \vdash k(\Gamma) \Imp k(\varphi)$.

We show the `if' part by induction on the height of derivation of $k(\Gamma) \Imp k(\varphi)$ in $\msf{G3ip}$. Let $k(\Gamma) \Imp k(\varphi)$ be obtained by $(\mathtt{R})$.

(1) $(\mathtt{R})$ is an axiom. If $(\mathtt{R})$ is $(\bot)$, then $\bot$ must occur in $\Gamma$. Then we have $\msf{G3SDM}\vdash\Gamma\Imp\varphi$ by $(\bot\Imp)$. Assume that $(\mathtt{R})$ is $(Id)$. Then $k(\varphi)\in \Xi$. By the definition of $k$, $\varphi$ must occur in $\Gamma$. Hence $\msf{G3SDM}\vdash\Gamma\Imp\varphi$.

(2) $(\mathtt{R})$ is $(\wedge\Imp)$. Let the conclusion be $\psi_1 \wedge \psi_2, k(\Gamma') \Imp k(\varphi)$ and the premiss be $\psi_1, \psi_2, k(\Gamma') \Imp k(\varphi)$. We have three cases:

(2.1) $\psi_1\wedge\psi_2 = k(\psi_1')\wedge k(\psi_2') = k(\psi_1' \wedge \psi_2')$. By induction hypothesis, $\msf{G3SDM}\vdash \psi_1', \psi_2', \Gamma\Imp \varphi$. By $(\wedge\Imp)$, $\msf{G3SDM}\vdash \psi_1'\wedge \psi_2', \Gamma\Imp \varphi$.

(2.2) $\psi_1 \wedge \psi_2 = k(\lnot\psi'_1) \wedge k(\lnot\psi'_2) = k(\lnot(\psi'_1\vee\psi'_2))$. By induction hypothesis, $\msf{G3SDM} \vdash \lnot\psi'_1, \lnot\psi'_2, \Gamma' \Imp \varphi$. By $(\lnot\vee\Imp)$,$\msf{G3SDM} \vdash \lnot(\psi'_1\vee\psi'_2), \Gamma' \Imp \varphi$.

(2.3) $\psi_1 \wedge \psi_2 = k(\lnot\lnot\psi'_1) \wedge k(\lnot\lnot\psi'_2) = k(\lnot\lnot(\psi'_1\wedge\psi_2'))$. By induction hypothesis, $\msf{G3SDM} \vdash \lnot\lnot\psi'_1, \lnot\lnot\psi'_2, \Gamma' \Imp \varphi$. By $(\lnot\lnot\wedge\Imp)$, $\msf{G3SDM} \vdash \lnot\lnot(\psi'_1\wedge\psi'_2), \Gamma' \Imp \varphi$. 

(3) $(\mathtt{R})$ is $(\Imp\wedge)$, $(\vee\Imp)$, or $(\Imp\vee)$. The proof is similar to (2).
\end{proof}

\begin{example}
Note that $\msf{G3SDM}\vdash \lnot(\lnot p\wedge \lnot q)\Leftrightarrow \lnot(\lnot\lnot\lnot p\wedge \lnot\lnot\lnot q)$. According to the definition $k$, we have $k(\lnot(\lnot p\wedge \lnot q)) = k(\lnot\lnot(p\vee q))$ and
$k(\lnot(\lnot\lnot\lnot p\wedge \lnot\lnot\lnot q)) = k(\lnot\lnot(\lnot\lnot p\vee \lnot\lnot q))
= k(\lnot(\lnot p\wedge \lnot q))
= k(\lnot\lnot(p\vee q))$.
Then $k(\lnot(\lnot p\wedge \lnot q)) = k(\lnot(\lnot\lnot\lnot p\wedge \lnot\lnot\lnot q))$. Clearly $\msf{G3ip}\vdash k(\lnot(\lnot p\wedge \lnot q))\Rightarrow k(\lnot(\lnot\lnot\lnot p\wedge \lnot\lnot\lnot q))$ and $\msf{G3ip}\vdash k(\lnot(\lnot\lnot\lnot p\wedge \lnot\lnot\lnot q))\Rightarrow k(\lnot(\lnot p\wedge \lnot q))$.
\end{example}

\subsection{Embedding of $\msf{G3DM}$ into $\msf{G3ip}$+\textbf{Gem-at}}
The language of classical propositional logic $\msf{CL}$ consists of the set of propositional variables $\Xi_2$, constants $\bot$, and propositional connectives $\wedge, \vee$ and $\supset$.
The set of all CL-terms $\mc{T}^{C}$ is defined inductively as the same as Int-terms except that propositional variables are only from $\Xi_2$. A CL-structure is a finite multiset of CL-terms. CL-structures are denoted by $X$, $Y$, $Z$ etc. A CL-sequent is an expression $X\Imp \theta$ where $X$ is a CL-structure and $\theta$ is a CL-term.
Following von Plato (cf.~\cite[p. 115]{negr:stru01}), 
the single-succedent sequent calculus for classical propositional logic $\msf{G3ip}$+\textbf{Gem-at} is obtained from $\msf{G3ip}$ by adding the following rule of excluded middle:
\[
\AxiomC{$p, X\Imp \theta$\quad$\neg p,X\Imp \theta$}
\RightLabel{\small $(Gem$-$at)$}
\UnaryInfC{$X\Imp \theta$}
\DisplayProof
\]
We shall show that $\msf{G3DM}$ can be embedded into $\msf{G3ip}$+\textbf{Gem-at}. Following the idea in Kamide \cite{kami:note11}, one can define such a translation $h$ as follows.

\begin{definition}
The translation $h: \mc{T} \imp \mc{T}^C$ is defined as follows:
\begin{align*}
h(p) &= p & h(\lnot p) &= p'\\
h(\bot) &= \bot & h(\lnot\bot) &= \top\\
h(\varphi \wedge \psi) &= h(\varphi) \wedge h(\psi) & h(\lnot(\varphi \wedge \psi)) &= h(\lnot\varphi) \vee h(\lnot\psi)\\
h(\varphi \vee \psi) &= h(\varphi) \vee h(\psi) & h(\lnot(\varphi \vee \psi)) &= h(\lnot\varphi) \wedge h(\lnot\psi)\\
h(\lnot\lnot\varphi) &= h(\varphi)
\end{align*}
For any structure $\Sigma = (\varphi_1, \ldots, \varphi_n)$, define $h(\Sigma) = (h(\varphi_1), \ldots, h(\varphi_n))$. 
\end{definition}

\begin{theorem}
For any DM-sequent $\Sigma\Imp\varphi$, $\msf{G3DM} \vdash \Sigma \Imp \varphi$ if and only if $\msf{G3ip}\mrm{+\textbf{\em Gem-at}}\vdash h(\Sigma) \Imp h(\varphi)$. 
\end{theorem}
\begin{proof}
Both directions are shown by induction on the height of derivation. The proof is similar with the proof of Theorem \ref{thm:emb_into_int}.
\end{proof}

\subsection{The Diagram of Embeddings}
It is well-known that the classical propositional logic is embedded into intuitionistic propositional logic by Glivenko's double negation translation. Let $g:\mc{T}^C\imp \mc{T}^I$ be the translation $g(\theta) = \neg\neg \theta$. For any structure $X=(\theta_1,\ldots,\theta_n)$, define $g(X) = (\neg\neg \theta_1,\ldots,\neg\neg \theta_n)$. Then Glivenko's theorem can be stated as follows (cf.~\cite[p. 119]{negr:stru01}):

\begin{theorem}[Glivenko]
For any CL-sequent $X\Imp\theta$,  $\msf{G3ip}\mrm{+\textbf{\em Gem-at}}\vdash X\Imp \theta$ if and only if $\msf{G3ip} \vdash g(X)\Imp \neg\neg\theta$.
\end{theorem}

Putting the above translations $f, k, t, g$ all together, we obtain the following diagram of embeddings:
\[
\begin{tikzpicture}
\draw (0,0) node {$\bullet$} ;
\draw (-0.5,0) node[below] {$\msf{G3SDM}$};
\draw (0,1.5) node {$\bullet$} ;
\draw (-0.5,1.5) node[above] {$\msf{G3DM}$};
\draw [->] (0,1.5) -- (0,0.1) ;
\draw (-0.2,0.75) node {$f$};
\draw (1.5,1.5) node {$\bullet$} ;
\draw (2.5,1.5) node[above] {$\msf{G3ip}$+\textbf{Gem-at}};
\draw [->] (1.5,1.5) -- (1.5,0.1) ;
\draw (1.5,0) node {$\bullet$} ;
\draw (1.7,0.75) node {$g$};
\draw (1.7,0) node[below] {$\msf{G3ip}$};
\draw [->] (0,0) -- (1.4,0) ;
\draw (0.75,0.2) node {$k$};
\draw [->] (0,1.5) -- (1.4,1.5) ;
\draw (0.75,1.3) node {$h$};
\end{tikzpicture} 
\]
This diagram commutes in the following sense: for any DM-sequent $\Gamma\Imp\varphi$, $\msf{G3ip}\vdash g\circ h(\Gamma)\Imp g\circ h(\varphi)$ if and only if $\msf{G3ip}\vdash k\circ f(\Gamma)\Imp k\circ f(\varphi)$.

\section{Conclusion}
The proof-theoretic study on semi-De Morgan and De Morgan algebras in the present paper has three main contributions. First, a sequent calculus for semi-De Morgan algebras is established. The decidability of derivability and the Craig interpolation are obtained proof-theoretically. Second, a single-succedent sequent calculus for De Morgan algebras is established, and the decidability and Craig interpolation are also shown proof-theoretically. Third, we develop some embedding theorems between sequent calculi. The G\"odel-Gentzen translation and hence Glivenko's double negation translation from classical to intuitionistic propositional logics are extended to the sequent calculi for De Morgan and semi-De Morgan algebras.

There are several directions for future work based on the results in the present paper. For example, the Kripke semantics for $\msf{G3SDM}$ and its modal extensions is not known. In a more broad setting on the study of negation, the negation in semi-De Morgan algebras provides a new member of the family of negations. The study on negation can be extended to it.

\paragraph{Acknowledgements.} 
This work is supported by Chinese National Fundation of Social Sciences and Humanities (grant no. 16CZX049).




\end{document}